\documentclass[10pt]{amsart}
\usepackage{amssymb,amsmath,amsthm,amsfonts}
\usepackage{graphicx}
\usepackage{euscript,mathrsfs}
\usepackage{xcolor}
\usepackage{color}
\usepackage[T1]{fontenc}
\usepackage[utf8]{inputenc}
\usepackage{latexsym}
\usepackage{mathabx}
 
\theoremstyle{plain}
\newtheorem{theorem}{Theorem}
\newtheorem{proposition}[theorem]{Proposition}
\newtheorem{corollary}[theorem]{Corollary}

\theoremstyle{definition}
\newtheorem{definition}[theorem]{Definition}

\newtheorem{remark}[theorem]{Remark}
\newtheorem{question}[theorem]{Question}
\newtheorem{example}[theorem]{Example}

\DeclareMathOperator{\Imag}{Im}

\DeclareMathOperator{\side}{side}
\DeclareMathOperator{\diam}{diam}
\DeclareMathOperator{\dist}{dist}

\DeclareMathOperator{\GL}{GL}
\DeclareMathOperator{\Capp}{Cap}
\DeclareMathOperator{\spt}{spt}
\DeclareMathOperator{\Exc}{Exc}
\DeclareMathOperator{\GQC}{GQC}
\DeclareMathOperator{\SG}{SG}
\DeclareMathOperator{\SC}{SC}

\newcommand{\loc}{{\scriptstyle{loc}}}

\newcommand{\Sob}{{\scriptstyle{Sob}}}
\newcommand{\RH}{{\scriptstyle{RH}}}

\newcommand{\field}{\mathbb}
\newcommand{\R}{{\field{R}}}
\newcommand{\C}{{\field{C}}}

\newcommand{\N}{{\field{N}}}
\newcommand{\Sph}{{\field{S}}}

\newcommand{\bA}{{\mathbf A}}

\newcommand{\cK}{{\mathcal K}}
\newcommand{\cL}{{\mathcal L}}
\newcommand{\cF}{{\mathcal F}}
\newcommand{\cH}{{\mathcal H}}
\newcommand{\cM}{{\mathcal M}}
\newcommand{\eps}{{\epsilon}}
\newcommand{\ba}{{\mathbf a}}
\newcommand{\by}{{\mathbf y}}
\newcommand{\bi}{{\mathbf i}}
\newcommand{\cX}{{\mathcal X}}

\newcommand{\restrict}{\begin{picture}(12,12)
                        \put(2,0){\line(1,0){8}}
                        \put(2,0){\line(0,1){8}}
                       \end{picture}}

\makeindex             

\numberwithin{theorem}{section}
\numberwithin{equation}{section}

\title[Quasiconformal and Sobolev distortion of dimension]{Quasiconformal and Sobolev \\ distortion of dimension}
\author{Jeremy T. Tyson}
\address{Department of Mathematics, University of Illinois at Urbana-Champaign, 1409 West Green St., Urbana, IL 61801, USA}
\email{tyson@illinois.edu}
\date{July 30, 2025}

\begin{document}
\maketitle

\begin{abstract}
We review a selection of the literature on the distortion of metric notions of dimension under quasiconformal, quasisymmetric, and Sobolev mappings. Our story begins with Gehring's landmark 1973 higher integrability theorem for quasiconformal maps, along with its implications for the distortion of Hausdorff dimension. Astala's 1994 solution to the planar higher integrability conjecture led to renewed interest in the subject in two dimensions. We continue with results from the 2000s and 2010s on the distortion of dimension by Sobolev maps, including estimates for dimension increase for generic elements in parameterized families of subsets. In the abstract metric setting, Pansu's notion of conformal dimension provides a key quasisymmetric invariant which has been useful in a wide range of applications. We briefly review relevant facts about conformal dimension, highlighting results of interest in the Euclidean setting. We conclude with recent work of the author in collaboration with Chrontsios Garitsis and with Fraser, extending the previous theory to interpolating dimensions and providing new insight into both quasiconformal classification and conformal dimension.
\end{abstract}

\section{Introduction}\label{sec:introduction}

In this survey article, we discuss the distortion of metric notions of dimension by quasiconformal and Sobolev mappings in Euclidean domains.

The origins of this line of research lie in foundational analytic properties of quasiconformal maps, especially, higher integrability of the Jacobian as established by Bojarski (in the planar case), and Gehring (in all dimensions). A celebrated advance by Astala in the mid-1990s identified the sharp higher integrability exponent for planar quasiconformal maps, leading to renewed interest in their dimension and measure-theoretic distortion behavior. Subsequently, attention shifted to general Sobolev maps, leading to a new array of questions concerning generic dimension distortion for typical members in parameterized families of sets as well as analogous theories in non-Euclidean and metric settings.

The field of analysis in metric spaces has developed rapidly since its inception in the late 1990s. In this setting, the metric notion of quasisymmetric mapping is nowadays recognized as the natural counterpart for the classical Euclidean concept of quasiconformal mapping, and spaces of abstract first-order Sobolev mappings play a fundamental role in the equivalence of metric, geometric and analytic definitions for such mappings. Many classical dimension distortion results have metric space analogs. In concert with the development of analysis in metric spaces, substantial interest lies in dilatation-free (lower) bounds on quasisymmetric dimension distortion, formalized in Pansu's notion of {\it conformal dimension}.

A wide array of metric notions of dimension have been introduced to illuminate different features of fractal sets and spaces. Hausdorff dimension and box-counting dimension are by far the most classical notions, and our survey begins with them. With the rise of analysis in metric spaces, Assouad dimension has also surfaced as a notion of key interest. We briefly explain why, and we highlight how distortion results for Assouad dimension differ from those for Hausdorff and box-counting dimension. In the last section, we consider various dimension interpolants introduced by Falconer, Fraser, and their collaborators. Such concepts are one-parameter families of dimension functions continuously interpolating between two given notions of metric dimension. The focus of that final section is on the distortion of two such dimension interpolants by Sobolev and quasiconformal maps.

We have tried to keep this survey nontechnical, avoiding detailed proofs. However, we relax this constraint in the case of one signature result: Gehring and V\"ais\"al\"a's foundational theorem on the quasiconformal distortion of Hausdorff dimension. Here we provide more detail about their proof, and we sketch or comment on several other proofs using different tools. In this way, we showcase the scope and power of the underlying theory and draw out new connections for the reader to explore.

This survey is aimed at researchers in fractal geometry and geometric measure theory with limited prior exposure to quasiconformal mappings. In order to keep the presentation widely accessible, we restrict attention to Euclidean space. Some results which we present are more properly part of the abstract theory of analysis in metric spaces, but retain novelty and interest when restricted to the classical Euclidean case. For instance, our discussion of conformal dimension highlights specifically Euclidean consequences of a rich metric theory about which we are (largely) silent. We provide copious references to the literature to guide the reader in exploring these ideas further. In particular, sections \ref{sec:qc-history-and-background}, \ref{sec:conformal-dimension}, and \ref{sec:interpolation} point the reader to other surveys, books, and articles of general interest which cover the basics of quasiconformal maps, analysis in metric spaces and conformal dimension, and dimension interpolation.

Due to space constraints, we omit any discussion of numerous other interesting avenues of contemporary research, including
\begin{itemize}
\item dimension distortion results specific to the sub-Riemannian setting;
\item applications of dimension distortion in dynamics (especially holomorphic dynamics), Kleinian groups, geometric group theory, and analysis on fractals;
\item conformal dimensions and quasiconformal classification of self-affine Euclidean fractals;
\item the theory of quasisymmetric and Sobolev mappings in metric measure spaces.
\end{itemize}

\subsection{Basic definitions and notation}\label{subsec:basic-definitions-and-notation}

For real numbers $a$ and $b$, we write $a \wedge b := \min\{a,b\}$ and $a \vee b := \max\{a,b\}$.

Invertible linear mappings $f$ of $\R^n$ are in bijective correspondence with matrices $\bA \in \GL(n,\R)$. We denote the operator norm of a matrix $\bA \in M(n,\R)$ by $||\bA||:= \max\{|\bA x|:|x|=1\}$. This is the {\it maximal stretching factor} for the map $f(x) = \bA x$. When $\bA \in \GL(n,\R)$ we also make use of the {\it minimal stretching factor} $\ell(\bA) = \min\{|\bA x|:|x|=1\} > 0$. The determinant of $\bA$ is denoted $\det \bA$.

For sets $E,F$ in a metric space $(X,d)$ we write $\diam E = \sup_{x,y\in E} d(x,y)$ and $\dist(E,F) = \inf_{x\in E, y\in F} d(x,y)$. Write $\side(Q)$ for the side length of a cube $Q \subset \R^n$. The Lebesgue measure of a set $A \subset \R^n$ will be denoted $|A|$.

The {\it Hausdorff dimension} $\dim_H E$ of $E \subset \R^n$ is the infimum of $s>0$ so that for each $\eps>0$ there exists a cover $\{A_i\}$ of $E$ with $\sum_i (\diam A_i)^s < \eps$. Alternatively, $\dim_H E$ is the unique $s_0$ for which $\cH^s(E) = 0$ for all $s \in (s_0,n]$ and $\cH^s(E) = +\infty$ for all $s \in [0,s_0)$, where $\cH^s$ denotes the $s$-dimensional Hausdorff measure.

For bounded $E$, the {\it (upper) box-counting dimension} $\dim_B E$ is the infimum of $s>0$ for which there exists $C>0$ so that $N(E,r) \le C r^{-s}$ for all $r<\diam E$. Here $N(E,r)$ denotes the minimal number of sets in $\R^n$ of diameter $r$ needed to cover $E$. 

We denote by $\cM(\R^n)$ the collection of Radon measures in $\R^n$. The {\it $s$-energy} of $\nu\in\cM(\R^n)$ is $I_s(\nu):= \iint |x-y|^{-s} \, d\nu(x)\,d\nu(y)$. The traditional (volume growth) version of Frostman's lemma states that the Hausdorff dimension of a Borel set\footnote{The result is also true for analytic subsets.} $E \subset \R^n$ is equal to $\sup\{s:\exists\,\nu \in \cM(\R^n) \mbox{ s.t.\ } \nu(B(x,r)) \le r^s \, \forall \, x \in E, r>0 \}$. The energy version of Frostman's lemma asserts that $\dim_H E = \sup\{t:\exists\,\nu\in\cM(\R^n) \mbox{ s.t.\ } I_t(\nu) < \infty \}$. For proofs of these facts, see \cite[Chapter 8]{mat:gmt}.

A {\it self-similar iterated function system (IFS)} is a set of contractive similarities $\cF = \{\varphi_1,\ldots,\varphi_M\}$. The {\it invariant set} of $\cF$ is the unique nonempty compact set $K$ with $K = \bigcup_{j=1}^M \varphi_j(K)$. The existence and uniqueness of $K$ follows from Banach's contraction mapping principle applied to the map $F \mapsto \bigcup_{j=1}^M \varphi_j(F)$ acting on the {\it compacta hyperspace} $\cK(\R^n)$.
Hutchinson's theorem states that if $\cF$ satisfies the OSC,\footnote{$\cF$ satisfies the {\it open set condition (OSC)} if there exists $O$ open and bounded so that $\varphi_j(O) \subset O \, \forall \, j$ and $\varphi_j(O) \cap \varphi_k(O) = \emptyset \, \forall \, j \ne k$.} then $s=\dim_H K$ is the unique nonnegative solution to $\sum_{j=1}^M r_j^s = 1$. Here $r_j$ denotes the contraction ratio of $\varphi_j$. Moreover, $\cH^s \restrict K$ is an Ahlfors $s$-regular measure: there exist $c_1>0$ and $c_2<\infty$ so that $c_1 r^s \le \cH^s(B(x,r) \cap K) \le c_2 r^s$ for all $x \in K$ and $0<r<\diam K$. See, e.g., \cite[Theorem 4.14]{mat:gmt}.

\subsection{Acknowledgements}

Prior work of the author described in this survey was obtained while under the support of the U.S. National Science Foundation (NSF) and the Simons Foundation. In addition, the results on quasiconformal distortion of the Assouad spectrum were obtained during the author's service as an NSF Program Officer and under support of the NSF Independent Research/Development Program. Any opinion, findings, and conclusions or recommendations expressed in this material are those of the author and do not necessarily reflect the views of the National Science Foundation.

Figure \ref{fig:gaskets}(b) is reprinted with permission from the article \cite{tw:gasket}, where it appears as Figure 5(a). The unpublished Theorem \ref{th:ct-qc-assouad} relies on an observation communicated to the author and Chrontsios Garitsis by Dimitrios Ntalampekos. I thank him for permission to include the comment here.

The animation at {\tt https://publish.illinois.edu/jeremy-tyson/animations} was produced by University of Illinois at Urbana--Champaign students Dan Schultz and Nishant Nangia with support from the {\it Illinois Geometry Lab (IGL)}\footnote{now known as the {\it Illinois Mathematics Lab (IML)}} in the Department of Mathematics.

Over the years, I have benefited from valuable conversations on the subject of this article with numerous colleagues, including Kari Astala, Zolt\'an Balogh, Mario Bonk, Efstathios Chrontsios Garitsis, Jonathan Fraser, Hrant Hakobyan, Leonid Kovalev, Alex Rutar, Kevin Wildrick, and Jang-Mei Wu. Thanks are especially due to Efstathios Chrontsios Garitsis, Estibalitz Durand Cartagena, and Dmitrios Ntalampekos, who each provided useful feedback on an early draft of this paper. Any errors in the text, however, remain the author's responsibility.

\section{Quasiconformal mappings: history and background}\label{sec:qc-history-and-background}

Non-conformality of an orientation-preserving diffeomorphism $f:\Omega \to \Omega'$ between planar domains is quantified via the {\it Beltrami cofficient} $\mu = \mu_f:\Omega \to \C$, defined by
\begin{equation}\label{eq:beltrami}
f_{\overline{z}} = \mu \, f_z.
\end{equation}
The complex-valued quantity $\mu_f(z)$ measures the infinitesimal stretching and rotation of $f$ at $z \in \Omega$. In particular, $|\mu_f(z)|$ determines the infinitesimal relative stretching factor, and the diffeomorphism $f$ is said to be {\it $k$-quasiconformal}, $k<1$, if $k_f := ||\mu_f||_\infty \le k$. Geometrically, a map $f$ is quasiconformal if the image of every infinitesimal disc is an infinitesimal ellipse, subject to the constraint that the eccentricies of all such ellipses are uniformly bounded away from zero. When $\mu=0$, \eqref{eq:beltrami} reduces to the classical Cauchy--Riemann equations, whose diffeomorphic solutions are conformal.

The {\it maximal} and {\it minimal infinitesimal stretch factors} $L_f(z) = \lim_{r\to 0} r^{-1} L_f(z,r)$ and $\ell_f(z) = \lim_{r \to 0} r^{-1} \ell_f(z,r)$ are defined via $L_f(z,r) = \max_{|z-w|=r} |f(z)-f(w)|$ and $\ell_f(z,r) = \min_{|z-w|=r} |f(z)-f(w)|$. These values satisfy the identities $L_f(z) = \max_\theta |\partial_\theta f(z)| = |f_z(z)| + |f_{\overline{z}}(z)|$ and $\ell_f(z) = \min_\theta |\partial_\theta f(z)| = |f_z(z)| - |f_{\overline{z}}(z)|$.
The {\it (metric) dilatation} $H_f(z) := L_f(z)/\ell_f(z)$ of $f$ at $z$ is related to the Beltrami coefficient via $H_f(z) = \tfrac{1+|\mu_f(z)|}{1-|\mu_f(z)|}$. One also says that $f$ is {\it $K$-quasiconformal} if $H_f := \max_z H_f(z)\le K$. Note that $H_f = \tfrac{1+k_f}{1-k_f}$ or equivalently $k_f = \tfrac{H_f-1}{H_f+1}$. By an abuse of terminology, we declare a planar map $f$ to be either $K$-quasiconformal for $K \ge 1$ or $k$-quasiconformal for $0\le k < 1$ if the appropriate condition holds true.

While early work assumed $C^1$ regularity, Morrey \cite{mor:qc} established the existence of homeomorphic solutions $f:\Omega \to \Omega'$ to the equation \eqref{eq:beltrami} for arbitrary measurable data $\mu:\Omega \to \C$ with $||\mu||_\infty < 1$. This result now goes under the name of the {\it measurable Riemann mapping theorem}. Following important work of Ahlfors \cite{ahl:qc}, Mori \cite{mori:qc} and others in the 1950s, the essential analytic characteristics of these mappings were put in place. Such characteristics include local H\"older continuity, $L^2$ integrability of the partial derivatives, Lusin's condition (N) (preservation of the vanishing of Lebesgue area measure), a.e.\ positivity of the Jacobian determinant, and a suitable compactness theory. Particularly relevant for us is the fact that such maps preserve full Hausdorff dimension; if $\dim_H E = 2$, then also $\dim_H f(E) = 2$. In \cite{boj:higher}, Bojarski used Calder\'on--Zygmund singular integral machinery to solve \eqref{eq:beltrami}, and as a consequence obtained a fundamental regularity conclusion: the partial derivatives $f_z$ and $f_{\overline{z}}$ for any $k$-quasiconformal map $f$ lie in $W^{1,p}_\loc$ for some $p>2$.\footnote{In fact, the conclusion holds for any $p$ such that $k_f C_p < 1$, where $C_p$ denotes the operator norm of the {\it Beurling transform} $Tf(z) = -\tfrac1\pi \iint \tfrac{f(w)}{(w-z)^2} \, dA(w)$ acting on $L^p$. Existence of such $p$ stems from the continuity of $p \mapsto C_p$ and the fact that $T$ is an $L^2$ isometry.}

In higher dimensional Euclidean spaces, the close connections to complex analysis disappear and new methodology is required to develop a suitable and analogous theory. This program was advanced primarily through the efforts of Gehring and V\"ais\"al\"a starting in the 1960s. Building on ideas pioneered by Ahlfors in the two-dimensional case, Gehring \cite{geh:rings} and V\"ais\"al\"a \cite{vai:qc} developed the theory starting from a {\it geometric definition} of quasiconformality, based on quasi-preservation of the {\it conformal modulus} of a curve family. Nevertheless, this definition was shown to be equivalent to the {\it metric definition}, phrased in terms of a uniform bound on the metric dilatation $H_f(x) = \limsup_{r \to 0} L_f(x,r)/\ell_f(x,r)$
of a homeomorphism $f:\Omega \to \Omega'$ in $\R^n$, $n \ge 2$, at a point $x \in \Omega$. Such maps are known to enjoy suitable analytic regularity; they lie in $W^{1,n}_\loc(\Omega:\R^n)$ and satisfy the distortion inequality
\begin{equation}\label{eq:analytic-QC}
||Df(x)||^n \le K \, \det Df(x) \qquad \mbox{a.e.\ $x \in \Omega$.}
\end{equation}
We call a homeomorphism $f \in W^{1,n}_\loc(\Omega:\R^n)$ {\it (analytically) $K$-quasiconformal} if \eqref{eq:analytic-QC} holds true, and for simplicity we abbreviate quasiconformal as QC. Similar to the planar case, such maps were shown to be H\"older continuous, differentiable almost everywhere, have Jacobian determinant positive almost everywhere, and preserve full Hausdorff dimension. However, the question of higher integrability of the Jacobian (that is, $L^p$ integrability of $Df$ for some $p>n$) remained unresolved for some time. The solution to this problem by Gehring in 1973 \cite{geh:higher} begins our discussion of QC and Sobolev dimension distortion, and is deferred to the following section.

Restriction and extension theorems for QC maps feature prominently in the theory. Beurling and Ahlfors \cite{ba:extension} showed that each $k$-QC map $f:\Omega \to \Omega$ defined on the upper half-plane $\Omega = \{ z \in \C : \Imag(z) > 0 \}$ admits a homeomorphic extension to the closure of $\Omega$ whose restriction to the boundary, denoted $g:\R\to\R$, satisfies 
\begin{equation}\label{eq:M-condition}
\frac1M \le \frac{|g(x+t)-g(x)|}{|g(x)-g(x-t)|} \le M \qquad \forall \, x \in \R, t>0,
\end{equation}
with $M = M(k)$. The double inequality in \eqref{eq:M-condition} is known as the {\it $M$-condition}. Beurling and Ahlfors also showed that every homeomorphism $g:\R\to\R$ satisfying the $M$-condition for some $M<\infty$ extends to a $k$-QC map $f:\Omega \to \Omega$ with $k = k(M)$. 

Any QC map $f:\Omega \to \Omega$ of $\Omega = \{ x = (x_1,\ldots,x_{n+1}) \in \R^{n+1} : x_{n+1} > 0 \}$ extends to a homeomorphism of the closure of $\Omega$ whose restriction to $\R^n = \partial \Omega$ is a QC self-map of $\R^n$, and Tukia and V\"ais\"al\"a \cite{tv:extension} provided the higher-dimensional version of the Beurling--Ahlfors extension theorem by showing that every QC map $g:\R^n \to \R^n$ arises as the boundary values of a QC map $f:\Omega \to \Omega$. 

Each QC map $f$ between domains in $\R^n$, $n \ge 2$, is absolutely continuous with respect to Lebesgue measure and preserves full Hausdorff dimension ($\dim_H E = n$ implies $\dim_H f(E) = n$). However, Tukia \cite{tuk:haus-dim-qs} --- answering a question of Hayman and Hinkkanen --- showed that such conclusions need not hold true for maps of $\R$ satisfying the $M$-condition. In fact, the behavior with respect to Hausdorff dimension can be arbitrarily bad; there exist such maps $g:\R\to\R$ and sets $E \subset \R$ so that both $\dim_H (\R\setminus E)$ and $\dim_H E$ are arbitrarily small. See also Theorem \ref{th:qc-haus-dim-subsets-line}.

A homeomorphism $f:(X,d) \to (Y,d')$ between metric spaces is said to be {\it $\eta$-quasisymmetric} ($\eta$-QS), for some increasing homeomorphism $\eta:[0,\infty) \to [0,\infty)$, if $d'(f(x),f(y)) \le \eta(t) \, d'(f(x),f(z))$ whenever $x,y,z \in X$ and $t>0$ satisfy $d(x,y) \le t \, d(x,z)$. Such a map is {\it weakly $H$-quasisymmetric} ($H$-WQS), $H>0$, if $d'(f(x),f(y)) \le H \, d'(f(x),f(z))$ whenever $d(x,y) \le d(x,z)$. Every $\eta$-QS map is $H$-WQS with $H = \eta(1)$, and every $H$-WQS map between connected doubling spaces is $\eta$-QS for some $\eta$ depending only on $H$ \cite[Theorem 10.19]{hei:lams}. A homeomorphism $f:\Omega \to \Omega'$ between domains in $\R^n$, $n \ge 2$, is QC if and only if there exists $H$ so that $f$ is $H$-WQS in every ball $B(x,\tfrac12 \delta(x))$, where $\delta(x) := \dist(x,\partial\Omega)$. In particular, $f:\R^n \to \R^n$ is QC if and only if it is QS. Moreover, a homeomorphism $f:\R \to \R$ is QS (as defined in this paragraph) if and only if $f$ satisfies the $M$-condition.

\subsection{References}

The classic texts by Lehto and Virtanen \cite{lv:qc} and Ahlfors \cite{ahl:book} provide a detailed picture of planar QC mapping theory as of the early 1970s. The canonical reference for the higher-dimensional theory is V\"ais\"al\"a's monograph \cite{vai:book}, from which generations of researchers were trained. For a more recent perspective, the book by Gehring, Martin and Palka \cite{gmp:qc} is highly recommended. However, neither of these latter references discuss Gehring's higher integrability theorem or its consequences for dimension distortion. Iwaniec and Martin \cite{im:book} emphasize an analytic approach to mappings of bounded distortion in relation to non-linear analysis and elasticity theory; here one finds a treatment of higher integrability in the (more general) context of quasiregular maps. For a comprehensive analytic presentation of the planar theory of quasiconformal and quasiregular maps, see Astala, Iwaniec and Martin \cite{aim:qc}.

\section{Quasiconformal dimension distortion: the story begins}\label{sec:qc}

The systematic study of quasiconformal and Sobolev dimension distortion originates in two 1973 papers. The first \cite{geh:higher}, by Fred Gehring, contained the long-sought higher dimensional extension of Bojarski's higher integrability theorem from \cite{boj:higher}.

\begin{theorem}[Bojarski, $n=2$; Gehring, $n \ge 3$]\label{th:qc-higher-integrability}
For each $K \ge 1$ and $n \ge 2$, there exists $p = p(n,K) > n$ so that if $f:\Omega \to \Omega'$ is a $K$-quasiconformal mapping between domains in $\R^n$, then $f \in W^{1,p}_\loc(\Omega:\R^n)$.
\end{theorem}

As a consequence of Bojarski's higher integrability theorem, Gehring and V\"ais\"al\"a \cite{gv:hausdorff} obtained estimates for the quasiconformal distortion of the Hausdorff dimension of planar sets. The extension to higher dimensions was presented in \cite{geh:higher} as a corollary of Gehring's corresponding higher integrability theorem in all dimensions.

\begin{theorem}[Gehring--V\"ais\"al\"a, $n=2$; Gehring, $n \ge 3$]\label{th:qc-haus-dim-distortion}
Let $f: \Omega \to \Omega'$ be a $K$-quasiconformal map between domains in $\R^n$, $n \ge 2$, and $E \subset \Omega$ a set with $s = \dim_H E \in (0,n)$. Then there exist $c_1$ and $c_2$, depending on $K$, $n$, and $s$ so that
\begin{equation}\label{eq:qc-haus-dim-distortion}
0<c_1 \le \dim_H f(E) \le c_2 < n.
\end{equation}
\end{theorem}

Gehring and V\"ais\"al\"a also showed that \eqref{eq:qc-haus-dim-distortion} is sharp in the following sense: {\it for each $0<s,s'<n$ there exist compact sets $E$ and $E'$ in $\R^n$ with $\dim_H E = s$ and $\dim_H E' = s'$ and a QC map $f:\R^n \to \R^n$ so that $f(E) = E'$.} It is worthwhile to recall the construction, as it illustrates techniques which will resurface later. 

\begin{example}\label{ex:gv}
Fix $0<s,s'<n$. Let $Q = [0,1]^n$ and $Q_1,\ldots,Q_M \subset Q$ be disjoint, closed cubes satisfying $\sum_{i=1}^M \side(Q_i)^s = 1$.  Similarly, let $Q_1',\ldots,Q_M' \subset Q$ be disjoint, closed cubes satisfying $\sum_{i=1}^M \side(Q_i')^{s'} = 1$. We consider IFS $\cF := \{\varphi_1,\ldots,\varphi_M\}$ and $\cF' := \{\varphi_1',\ldots,\varphi_M'\}$ comprised of similarity maps $\varphi_i,\varphi_i'$ so that $\varphi_i(Q) = Q_i$ and $\varphi_i'(Q) = Q_i'$ for all $i$. Denote by $E$, resp.\ $E'$, the corresponding invariant sets. By Hutchinson's formula, $\dim_H E = s$ and $\dim_H E' = s'$. 

We now construct a QC map $f: \R^n \to \R^n$ with $f(E) = E'$. First, define $f_1:\R^n \to \R^n$ as follows: $f_1(x) = x$ for $x \not\in Q$, $f_1(x) = \varphi_{i}' \circ \varphi_i^{-1}(x)$ for each $x \in Q_i$ and $i=1,\ldots,M$, and $f_1$ is extended to the interior $U$ of $Q \setminus \bigcup_{i=1}^M Q_i$ in either a smooth or a piecewise linear fashion. It is easy to see that $f_1$ is $K$-QC for some fixed $K \ge 1$.\footnote{The value of $K$ depends on how $f_1$ is extended to $U$.} Moreover, $H_f(x) = 1$ provided $x \not\in Q$ or $x \in Q_i$ for some $i$. Next, define $f_2:\R^n \to \R^n$ by conjugating with the defining similarity maps: $f_2(x) = f_1(x)$ on $\R^n \setminus \bigcup_{i=1}^M Q_i$, and $f_2(x) = \varphi_i' \circ f_1 \circ {\varphi_i}^{-1}(x)$ for $x \in Q_i$. Then $f_2$ is again $K$-QC for the same $K$, and $H_f(x) = 1$ if $x \in f_j(Q_i)$ for some $i,j$. Continuing in this fashion, one constructs $K$-QC maps $f_m$ so that $f_m(\varphi_{i_1} \circ \cdots \circ \varphi_{i_m}(Q)) = \varphi_{i_1}' \circ \cdots \circ \varphi_{i_m}'(Q)$ for all words $(i_1 \cdots i_m) \in \{1,\ldots,M\}^m$. By a standard convergence theorem, there exists a limit map $f:\R^n \to \R^n$ which is again $K$-QC and maps $E$ onto $E'$.
\end{example}

\begin{example}\label{ex:snowflake}
The construction of the QC map in Example \ref{ex:gv} is made easier by the fact that the set $E$ satisfies a strong version of total disconnectivity: any two pieces $E \cap Q_i$ and $E \cap Q_j$, $i \ne j$, have mutual distance bounded from below by a fixed multiple of $\diam E$. Implementing a similar construction for connected invariant sets is more challenging. We illustrate by describing a QS equivalence between $[0,1] \subset \R^2$ and the von Koch snowflake. Recall that the {\it von Koch curve} $K \subset \R^2$ (Figure \ref{fig:vk}) can be realized as the invariant set for a self-similar IFS $\cF'$ with the OSC, consisting of four contractions each with scale factor $\tfrac13$. Then $\dim K = \tfrac{\log 4}{\log 3}$.

\begin{figure}[b]
\includegraphics[scale=.175]{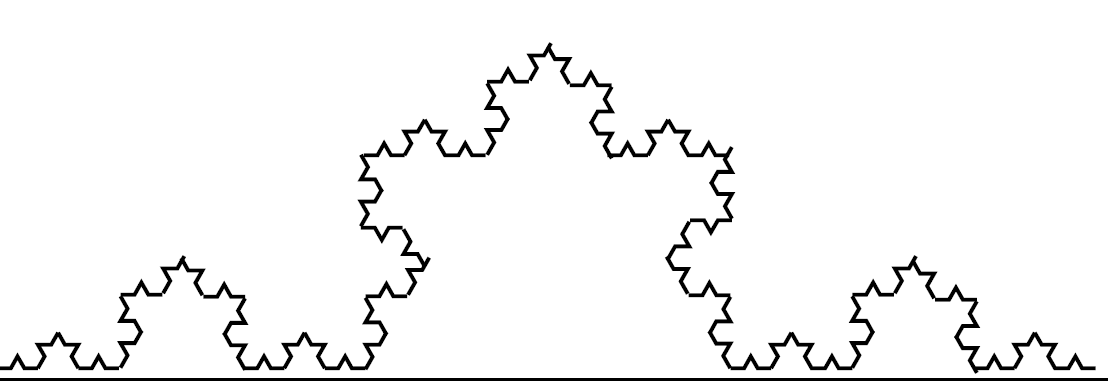}
%
%
\caption{The von Koch curve}
\label{fig:vk}
\end{figure}

To construct a QS map from $[0,1]$ to $K$, represent $[0,1]$ as the invariant set for an IFS $\cF$ consisting of four contractions each with scale factor $\tfrac14$. There exists a natural order-preserving identification between the elements of $\cF$ and $\cF'$. This identification descends to all levels of the IFS construction and induces a homeomorphism $f:[0,1] \to K$. Note that, for any $m \in \N$, $f$ maps an $m$th level self-similar piece of $[0,1]$ of diameter $4^{-m}$ onto an $m$th level piece of $K$ of diameter $3^{-m}$. In fact, setting $\eps := \tfrac{\log 3}{\log 4}$, one can verify that $f:([0,1],d_E^\eps) \to (K,d_E)$ is bi-Lipschitz. Since the identity map from $[0,1]$ to $([0,1],d_E^\eps)$ is QS, it follows that $f:[0,1] \to K$ is QS.

We also remark that a generalized version of the Beurling--Ahlfors extension theorem ensures that the map $f:[0,1] \to K$ extends to a QC self-map of $\R^2$. Thus $K$ is a {\it quasi-line segment}, i.e., a QC image of $[0,1]$.  The book by Gehring and Hag \cite{gh:quasidisk} is highly recommended for a detailed discussion of quasidiscs (quasiconformal images of the unit disc in $\C$), including numerous characterizations for these objects.
\end{example}

Example \ref{ex:gv} indicates that the constants $c_1(K,n,s)$ and $c_2(K,n,s)$ in \eqref{eq:qc-haus-dim-distortion} must satisfy $c_1(K,n,s) \to 0$ and $c_2(K,n,s) \to n$ as $K \to \infty$. Semi-explicit expressions for these constants can be given by making use of the conclusion of Theorem \ref{th:qc-higher-integrability}.

\begin{definition}
For $K \ge 1$ and $n \ge 2$, the {\it Sobolev higher integrability exponent} $p^\Sob(n,K)$ is the supremum of those values $p>n$ for which the conclusion of Theorem \ref{th:qc-higher-integrability} holds true for each $K$-quasiconformal map between domains in $\R^n$.
\end{definition}

With this notation in place, the estimates in \eqref{eq:qc-haus-dim-distortion} can be stated in the following form: for $E \subset \Omega$ with $\dim_H E = s$ and $\dim_H f(E) = s'$,
\begin{equation}\label{eq:qc-haus-dim-distortion-b}
\alpha_n(p^\Sob(n,K)) \left( \frac1s - \frac1n \right) \le \frac1{s'} - \frac1n \le \alpha_n(p^\Sob(n,K^{n-1}))^{-1} \left( \frac1s - \frac1n \right).
\end{equation}
Here for $p>n$, we have set $\alpha_n(p) = 1 - \tfrac{n}{p}$.

We next sketch the argument of Gehring--V\"ais\"al\"a for the proof of \eqref{eq:qc-haus-dim-distortion-b}. The $n$-dimensional version provided here comes from Gehring \cite{geh:higher}. Later in this survey we indicate other proofs for this result.

\begin{proof}[Proof of \eqref{eq:qc-haus-dim-distortion-b}]
Let $f$ and $E$ be as stated. We may assume without loss of generality that $E \subset V := B(x_0,\tfrac12\delta(x_0))$ for some $x_0 \in \Omega$. Recall that  $f|_V$ is weakly $H$-QS for some $H>0$ depending on $n$ and $K$. Fix a domain $U \subset V$ so that $E \subset U$.

Let $p \in (n,p^\Sob(n,K))$ so that $f \in W^{1,p}(U:\R^n)$. Fix $t \in (\dim_H E,n)$ so that $\cH^t(E) = 0$. Define $t' \in (0,n)$ by the formula $(1-\tfrac{n}{p})(\tfrac{1}{t}-\tfrac{1}{n}) = \tfrac{1}{t'}-\tfrac{1}{n}$. We will show that $\dim_H f(E) \le t'$; letting $t \searrow \dim_H E$ and $p \nearrow p^\Sob(n,K)$ completes the proof.

For $\eps>0$, cover $E$ by dyadic cubes $(Q_i)$ such that $\sum_i \side(Q_i)^t < \eps$. Denote by $x_i$ the center of $Q_i$ and set $r_i = \side(Q_i)$. Since $f|_U$ is uniformly continuous, we have that $\diam f(Q_i) < \eps'$ for some $\eps'(\eps)\to 0$. Set $L_i = L_f(x_i,r_i)$ and $\ell_i = \ell_f(x_i,r_i)$. The weak $H$-quasisymmetry of $f|_U$ implies that $L_i \le H \ell_i$ for all $i$. We estimate
\begin{equation}\label{eq:proof-1}
\diam f(Q_i) \le 2 L_i \le 2H \ell_i \le C(n,H) |f(Q_i)|^{1/n}
\end{equation}
and
\begin{equation}\label{eq:proof-2}
|f(Q_i)| = \int_{Q_i} Jf \le \int_{Q_i} ||Df||^n \le \left( \int_{Q_i} ||Df||^p \right)^{n/p} \, |Q_i|^{1-n/p}.
\end{equation}
Using \eqref{eq:proof-1}, \eqref{eq:proof-2} and H\"older's inequality, we bound $\sum_i \diam f(Q_i)^{t'}$ from above by
\begin{equation}\begin{split}\label{eq:proof-3}
C(n,p,H,t') \left( \sum_i |Q_i|^{t'(\tfrac1n-\tfrac1p)(\tfrac{p}{p-t'})} \right)^{1-t'/p} \, \left( \int_{\cup_i Q_i} ||Df||^p \right)^{t'/p}.
\end{split}\end{equation}
From the choice of $t'$ we have $t'(\tfrac1n-\tfrac1p)(\tfrac{p}{p-t'}) = \tfrac{t}{n}$ and hence the sum in \eqref{eq:proof-3} equals $\sum_i |Q_i|^{t/n} = \sum_i r_i^t$. Thus $\cH^{t'}_{\eps'}(f(E)) \le \sum_i (\diam f(Q_i))^{t'}  \le C \eps^{1-t'/p} \, ||Df||_{L^p(U)}^{t'}$ with $C = C(n,p,H,t')$, and since $f \in W^{1,p}(U)$ and $\eps'(\eps) \stackrel{\eps \to 0}\longrightarrow 0$, we conclude that $\dim_H f(E) \le t'$ as desired. This establishes the left hand inequality in \eqref{eq:qc-haus-dim-distortion-b}; the right hand inequality follows since the inverse of a $K$-QC map in $\R^n$ is $K^{n-1}$-QC.
\end{proof}

The exact value of $p^\Sob(n,K)$ remains a challenging open problem when $n \ge 3$, but $p^\Sob(2,K)$ was computed in a groundbreaking result of Astala \cite{ast:2d} in 1994.

\begin{theorem}[Astala]\label{th:qc-higher-integrability-and-haus-dim-distortion-2d}
For each $K \ge 1$, we have $p^\Sob(2,K) = \tfrac{2K}{K-1}$. Consequently, if $E \subset \Omega$ and $f:\Omega \to \Omega'$ is a $K$-quasiconformal mapping between planar domains and $E \subset \Omega$ satisfies $\dim_H E = s$ and $\dim_H f(E) = s'$, then
\begin{equation}\label{eq:qc-haus-dim-distortion-b-2d}
\frac1K \left( \frac1s - \frac12 \right) \le \frac1{s'} - \frac12 \le K \left( \frac1s - \frac12 \right).
\end{equation}
\end{theorem}

Astala's proof uses the connection between quasiconformality and holomorphic motions characteristic of the planar case, along with the thermodynamic formalism for dimensions of invariant sets of conformal dynamical systems. By a careful analysis of holomorphically varying self-similar Cantor sets, inspired by the construction in Example \ref{ex:gv}, Astala also shows that the estimates in \eqref{eq:qc-haus-dim-distortion-b-2d} are sharp. To see that the value of $p^\Sob(2,K)$ is sharp, it suffices to note that the $K$-QC radial stretch map $f(z) = |z|^{1/K-1}z$ lies in the local Sobolev space $W^{1,p}_\loc(\C:\C)$ for all $p<\tfrac{2K}{K-1}$.

Astala's paper \cite{ast:2d} has had a profound influence on the further development of the theory of planar quasiconformal maps. As of the completion of this survey, MathSciNet listed over 200 papers referencing \cite{ast:2d}. We highlight several key strands of later research directly related to the estimates in \eqref{eq:qc-haus-dim-distortion-b-2d}. Building on earlier work of Astala--Clop--Mateu--Orobitg--Uriarte-Tuero \cite{acmout:distortion}, Lacey, Sawyer and Uriarte-Tuero \cite{lsut:borderline} established a borderline absolute continuity estimate associated to \eqref{eq:qc-haus-dim-distortion-b-2d}. Namely, if $E \subset \Omega \subset\C$ has $\cH^s(E) = 0$ for some $0<s<2$ and $f$ is $K$-QC, then $\cH^{s'}(f(E)) = 0$ for $s'$ such that $\tfrac1K(\tfrac1s-\tfrac12) = \tfrac1{s'}-\tfrac12$.

While Astala's original work established sharpness of the estimates in \eqref{eq:qc-haus-dim-distortion-b-2d} for all $s \in (0,2)$, the relevant examples involve self-similar Cantor sets. For $1$-dimensional sets $E \subset \C$, the upper bound for $\dim_H f(E)$ in \eqref{eq:qc-haus-dim-distortion-b-2d} takes the form
\begin{equation}\label{eq:qc-bound-quasilines-0}
\dim_H f(E) \le \frac{2K}{K+1};
\end{equation}
expressed in terms of the complex dilatation parameter $k = \tfrac{K-1}{K+1}$ \eqref{eq:qc-bound-quasilines-0} reads
\begin{equation}\label{eq:qc-bound-quasilines-1}
\dim_H f(E) \le 1 + k.
\end{equation}
It was long known, however, that the estimate in \eqref{eq:qc-bound-quasilines-1} was not best possible for quasilines, i.e., the case $E = \R$. Becker and Pommerenke \cite{bp:quasicircles} proved the estimate 
\begin{equation}\label{eq:qc-bound-quasilines-2}
\dim_H f(\R) \le 1 + C k^2
\end{equation}
with $C = 37$, and Smirnov \cite{sm:quasilines} answered a longstanding question by establishing \eqref{eq:qc-bound-quasilines-2} with $C = 1$.

\begin{theorem}[Smirnov]\label{th:haus-dim-quasilines}
Let $f: \C \to \C$ be $k$-quasiconformal for some $k<1$. Then $\dim_H f(\R) \le 1+k^2$.
\end{theorem}

The following result of Prause and Smirnov \cite{ps:qs-distortion-spectrum} builds on Tukia's example in the context of Theorem \ref{th:haus-dim-quasilines}.

\begin{theorem}[Prause--Smirnov]\label{th:qc-haus-dim-subsets-line}
Let $f: \R \to \R$ be QS with a $k$-QC extension to $\C$. Let $E \subset \R$ satisfy $\dim_H E = s \in [0,1]$. Then, setting $l = (1-s)^{1/2}$, we have
\begin{equation}\label{eq:qc-haus-dim-subsets-line}
\frac{(1-k^2)(1-l^2)}{(1+kl)^2} \le s' = \dim_H f(E) \le \frac{(1-(k \wedge l)^2)(1-l^2)}{(1-(k \wedge l) \, l)^2}.
\end{equation}
\end{theorem}


We also remark that Fuhrer, Ransford and Younsi \cite{fry:motions} have given a new proof of Astala's bounds \eqref{eq:qc-haus-dim-distortion-b-2d} via the theory of planar inf-harmonic functions. Their proof applies also to the box-counting and packing dimensions. We will discuss distortion of box-counting dimension in the following section, in the context of Sobolev maps.

Finally, we briefly review another proof for the estimates \eqref{eq:qc-haus-dim-distortion-b}. In \cite{im:capacity}, Iwaniec and Martin establish \eqref{eq:qc-haus-dim-distortion-b} by studying quasiconformal distortion of Sobolev capacity. Following \cite{im:capacity}, we recall that, for a compact set $E$ in $\Omega \subset \R^n$, the {\it variational $s$-Sobolev capacity} $s-\Capp(E,\Omega)$ is the infimum of the quantities $\int_\Omega |\nabla u|^s$ over all $u \in C^\infty_0(\Omega)$ such that $u$ is equal to one in a neighborhood of $E$. It is easy to see that the value of $s-\Capp(E,\Omega)$ is unchanged if the infimum is taken instead over $u \in W^{1,n}_\loc$ with compact support in $\Omega$. A closed set $E \subset \Omega$ is said to have {\it vanishing $s$-capacity} if $s-\Capp(E',U') = 0$ for every compact $E' \subset E$ and every $U' \subset \Omega$ with $E' \subset U'$. We cite without proof the following relationship between Hausdorff dimension and Sobolev capacity; see \cite{cos:capacity} for a proof in the setting of metric measure spaces.

\begin{proposition}\label{prop:haus-dim-sobolev-capacity}
If $E \subset \R^n$ is closed, then $\sup \{ s : \mbox{$E$ has vanishing $s$-capacity} \} = n - \dim_H E$.
\end{proposition}

A capacitary-theoretic proof of \eqref{eq:qc-haus-dim-distortion-b} follows by combining Proposition \ref{prop:haus-dim-sobolev-capacity} with the next proposition, a result of Iwaniec and Martin \cite[Lemma 1.1]{im:capacity}.

\begin{proposition}\label{prop:qc-invariant-sobolev-capacity}
Let $f:\Omega \to \Omega'$ be a QC map between domains in $\R^n$, and assume that $f$ lies in $W^{1,p}_\loc$ for some $p>n$. Let $E \subset \Omega$ be a closed set of vanishing $s$-capacity for some $0<s<n$. Then $f(E)$ has vanishing $r$-capacity for all $r \le \tfrac{s(p-n)}{p-s}$.
\end{proposition}

The proof of Proposition \ref{prop:qc-invariant-sobolev-capacity} involves a change of variable argument coupled with H\"older's inequality and Gehring's higher integrability theorem, similar to the previous argument of Gehring and V\"ais\"al\"a. Assuming without loss of generality that $E$ is compact, and choosing a test function $\varphi \in C^\infty_0(\Omega)$ for the $s$-capacity of $E$, one verifies that $\psi = \varphi \circ f^{-1}$ is admissible for the $r$-capacity of $f(E)$. Moreover,
\begin{equation*}\begin{split}
r-\Capp(f(E),\Omega') &\le \int_{\Omega'} |\nabla \psi|^r = \int_\Omega |\nabla \psi \circ f|^r \det Df \le \int_\Omega ||Df||^n |\nabla \psi \circ f|^r \\
&= \int_\Omega ||Df||^{n-r} \bigl( ||Df|| \cdot |\nabla \psi \circ f| \bigr)^r \le K^r \int_{\spt\varphi} ||Df||^{n-r} |\nabla \varphi|^r \\
\intertext{using the fact that $||Df|| \, |h| \le K |(Df)^T h|$ for a $K$-quasiconformal map $f$, together with the identity $(Df)^T (\nabla \psi \circ f) = \nabla \varphi$,}
&\le K^r \left( \int_{\spt\varphi} |\nabla \varphi|^s \right)^{r/s} \, \left( \int_{\spt\varphi} ||Df||^{s(n-r)/(s-r)} \right)^{1-r/s}.
\end{split}\end{equation*}
Since $\tfrac{s(n-r)}{s-r} \le p$, $\spt\varphi$ is compact, and $f \in W^{1,p}_\loc$, the second integral is finite. Taking the infimum over $\varphi$ implies that $f(E)$ has vanishing $r$-capacity.

\section{Sobolev dimension distortion: a new perspective}\label{sec:sobolev}

In the preceding section, we considered the distortion of Hausdorff dimension by quasiconformal homeomorphisms. The membership of such maps in Sobolev spaces with improved regularity played a key role. In this section we change perspective and consider mappings (not necessarily injective) in supercritical Sobolev classes. We observe that one-sided distortion estimates for Hausdorff dimension continue to hold; quasiconformality is not required. The line of research which we review in this section draws substantial inspiration from a 2000 paper by Kaufman \cite{kau:sobolev}. Kaufman begins his paper with an elementary proof for a one-sided distortion bound for the increase in Hausdorff dimension under a supercritical Sobolev mapping.

\begin{definition}
A mapping $f \in W^{1,p}(\R^n:\R^N)$ is {\it supercritical} if $p>n$.
\end{definition}

\begin{proposition}[Kaufman]\label{prop:kaufman}
Let $f \in W^{1,p}(\Omega:\R^N)$ be a continuous and supercritical mapping defined in a domain $\Omega \subset \R^n$. Let $E \subset \Omega$ be a set of $\sigma$-finite $\cH^s$ measure for some $s \in (0,n)$. Then $\cH^{s'}(f(E)) = 0$, where $s' = \tfrac{ps}{p-n+s}$. In particular,
\begin{equation}\label{eq:kaufman}
\dim_H f(E) \le \frac{p \, \dim_H E}{p-n+\dim_H E}.
\end{equation}
\end{proposition}

Theorem \ref{th:qc-haus-dim-distortion} follows immediately from Proposition \ref{prop:kaufman}, arguing (as before) first for fixed $p \in (n,p^\Sob(n,K))$, then letting $p \nearrow p^\Sob(n,K)$, and finally using the fact that the inverse of a $K$-quasiconformal map is $K^{n-1}$-quasiconformal.

\begin{proof}
We appeal to the fact that Hausdorff dimension can be computed via coverings by {\it dyadic cubes}. Without loss of generality, assume $\cH^s(E) < \infty$ and $E \subset U \Subset \Omega$. Fix $t> s = \dim_H E$ and let $\{Q_i\}$ be an essentially disjoint cover of $E$ by dyadic cubes in $U$ with $\sum_i (\diam Q_i)^t < \cH^s(E) + 1$. In place of the previous argument using relative metric distortion, we now appeal to the Morrey--Sobolev inequality: {\it there exists $C(n,p)>0$ so that $|f(x)-f(y)| \le C(n,p) |x-y|^{1-n/p} (\int_{Q_i} ||Df||^p )^{1/p}$ for all $x,y \in Q_i$.}\footnote{Here we rely on the fact that we consider the {\it continuous} representative of the given map.} In particular,
$\diam f(Q_i) \le C(n,p) (\diam Q_i)^{1-n/p} (\int_{Q_i} ||Df||^p )^{1/p}$ for all $i$. From here onwards, the proof closely follows the previous argument. With
$t' = \tfrac{pt}{p-n+t}$,
estimating $\sum_i (\diam f(Q_i))^{t'}$ from above via H\"older's inequality yields
$$
\sum_i (\diam f(Q_i))^{t'} \le C \left( \sum_i (\diam Q_i)^{t} \right)^{1-t'/p} \, \left( \int_{\cup_i Q_i} ||Df||^p \right)^{t'/p}.
$$
Since the cubes $\{Q_i\}$ were chosen to be essentially disjoint, we deduce that
$$
\cH^{t'}_{\eps'}(f(E)) \le C (\cH^s(E) + 1)^{1-t'/p} ||Df||_{L^p(U)}^{t'}.
$$
Since $s<n$, $E$ is a Lebesgue null set and we may choose the open set $U \supset E$ to have arbitrarily small Lebesgue measure. By continuity of the integral, and since $f$ lies in $W^{1,p}$, the domain $U$ can be chosen so that $||Df||_{L^p(U)}$ is as small as we please. It follows that $\cH^{t'}(f(E)) = 0$, hence $\dim_H f(E) \le t'$. Now let $t \searrow s$ (so $t' \searrow s'$).
\end{proof}

As stated, the argument only works for $\dim_H E < n$. However, an easy adaptation shows that $f\in W^{1,p}$ preserves the vanishing of the Hausdorff $n$-measure: if $E \subset \Omega$ has $\cH^n(E) = 0$, then $\cH^n(f(E)) = 0$. Moreover, $\dim_H f(E) \le n$ for any $E \subset \Omega$. 

Kaufman also proved that Proposition \ref{prop:kaufman} is sharp.

\begin{theorem}[Kaufman]\label{kau:sharpness}
Let $E \subset \R^n$ be any closed set with $\dim_H E \ge s$, and let $p>n$. Then there exists $f \in W^{1,p}(\R^n:\R^n)$ continuous so that $\dim_H f(E) \ge \tfrac{ps}{p-n+s}$.
\end{theorem}

The proof of Theorem \ref{kau:sharpness} is nonconstructive; a random family of continuous $W^{1,p}$ mappings $f_\omega:\R^n \to \R^n$ is constructed and it is shown via the energy version of Frostman's lemma that with positive probability one of the maps $f_\omega$ satisfies the stated lower bound. Later in this section, we discuss constructive proofs of Theorem \ref{kau:sharpness} and more general statements, for specific choices of the source set $E$.

Kaufman also proves analogous statements for box-counting and packing dimensions. We discuss the box-counting version. For an upper bound, analogous to Proposition \ref{prop:kaufman}, a new approach is needed since we now wish to estimate the covering number $N(f(E),r)$ for suitable scales $r>0$ rather than diameter sums of the form $\sum_i (\diam f(A_i))^t$. We again anticipate that the Morrey--Sobolev inequality will play a role, however, this inequality only provides an upper bound for the image diameter, which may be (and in many cases is) sub-optimal. 

To deal with these complications, Kaufman implements a stopping-time argument based on the size of the $f$-images of dyadic cubes. Fixing a scale $\delta = 2^{-m}$ and setting $r:= 2^{-ms/s'}$ where $s \in (0,n)$ is fixed and $s' = \tfrac{ps}{p-n+s}$, we declare a dyadic cube $Q \subset \Omega$ to be {\it $r$-major} if $\diam f(Q) \ge r$ and {\it $r$-minor} if $\diam f(Q) < r$. An $r$-minor cube $Q$ whose parent is $r$-major is {\it $r$-critical}. An easy counting argument using the Morrey--Sobolev inequality and the relationship between $r$ and $\delta$ shows that the number of $r$-major cubes in $\Omega$ is $\lesssim 2^{ms}$.\footnote{To obtain this conclusion, consider first only $r$-major dyadic cubes $Q \subset \Omega$ of side length $2^{-m'}$ for some $m' \ge m$, then sum the resulting inequality over all $m' \ge m$.} Consequently, the same bound holds true for the number of $r$-critical cubes. 

We start with a closed set $E \subset \Omega$ with $s = \dim_B E$ and choose $t>s$. Then $E$ can be covered by $\lesssim 2^{mt}$ dyadic cubes $\{Q_i\}$ of side length $\le 2^{-m}$. Moreover, $f(E)$ can be covered by the images $f(Q_i)$ of the $r$-minor cubes in $\{Q_i\}$ along with the images $f(Q)$ of all $r$-critical cubes $Q$ in $\Omega$. The resulting cardinality bound is $\lesssim 2^{ms} + 2^{mt} \lesssim 2^{mt} = r^{-t'}$. This yields the estimate $N(f(E),r) \lesssim C r^{t'}$ for sufficiently small $r>0$, and hence the desired estimate $\dim_B f(E) \le t' \searrow s' = \tfrac{ps}{p-n+s}$.

\medskip

The change of perspective from quasiconformal homeomorphism to general Sobolev mapping opens the door to numerous generalizations and extensions.
\begin{itemize}
\item The target space can be generalized; there is no need to consider equidimensional maps $f: \R^n \to \R^n$. In fact, in many subsequent results the geometry of the target space plays no role; results are naturally stated in terms of metric space-valued Sobolev spaces $W^{1,p}(\R^n:Y)$ for arbitrary $(Y,d_Y)$. To simplify the presentation in this survey, we restrict attention to Euclidean targets.
\item The previous results concerned {\it universal estimates}, i.e., estimates for the distortion of dimension of an arbitrary set $E$ under a given map $f$. We may instead consider families of subsets $E_\ba$ parameterized by elements $\ba$ in some measure space $W$ and seek {\it generic estimates} valid for almost every $\ba$. Such conclusions yield new information also in the quasiconformal setting.
\item Generalizations of the source space are also of interest. We may replace the Euclidean space $\R^n$ by, for instance, either a nonabelian Carnot group equipped with a sub-Riemannian metric, or a metric measure space with bounded geometry in the sense of Heinonen--Koskela.
\end{itemize}

We next turn to generic estimates for Sobolev dimension increase in parameterized families of sets. Fix a continuous and supercritical Sobolev map $f\in W^{1,p}(\Omega:\R^N)$, $\Omega \subset \R^n$, $p>n$. Let $V$ be an $m$-dimensional linear subspace of $\R^n$ with orthogonal complement $V^\perp$, and consider the foliation of $\Omega$ by affine spaces parallel to $V$ and indexed by elements of $V^\perp$. To wit, for $\ba \in V^\perp$, let $V_\ba := V + \ba$. We consider only those points $\ba \in V^\perp$ for which $V_\ba \cap \Omega \ne \emptyset$. For each such $\ba$, the set $V_\ba \cap \Omega$ has Hausdorff dimension $m$, whence $\dim_H f(V_\ba \cap \Omega) \le \tfrac{pm}{p-n+m}$ by Proposition \ref{prop:kaufman}. On the other hand, for almost every such point $\ba$ (with respect to $\cH^{n-m} \restrict V^\perp$), the restriction of $f$ to $V_\ba \cap \Omega$ is again a continuous, supercritical Sobolev map. By an earlier remark, $\dim_H f(V_\ba \cap \Omega) \le m$ for $\cH^{n-m}$-a.e.\ $\ba$.

In \cite{bmt:foliations}, Balogh, Monti and the author interpolated between these two statements, obtaining sharp estimates on the size of exceptional sets of parameters $\ba$ for which the dimension of $f(V_\ba \cap \Omega)$ exceeds some specified value $m<\alpha<\tfrac{pm}{p-n+m}$. We denote by $G(n,m)$ the Grassmannian of all $m$-dimensional linear subspaces of $\R^n$.

\begin{theorem}[Balogh--Monti--Tyson]\label{th:bmt-foliations}
Let $m \in \{1,\ldots,n-1\}$, and $V \in G(n,m)$ be given, and let $f \in W^{1,p}(\Omega:\R^N)$ be a continuous supercritical map. For each $m < \alpha < \tfrac{pm}{p-n+m}$, the set of $\ba \in V^\perp$ for which $\cH^\alpha(f(V_\ba \cap \Omega)) > 0$ is a null set for the Hausdorff measure $\cH^\beta \restrict V^\perp$. Here $\beta = \beta(\alpha) = (n-m) - p(1-\tfrac{m}{\alpha})$.
\end{theorem}

Of course, $\beta(m) = n-m$. We also note that $\beta(\tfrac{pm}{p-n+m}) = 0$. In \cite{bmt:foliations} the authors also showed that the result is sharp in the following sense.

\begin{theorem}[Balogh--Monti--Tyson]\label{th:bmt-sharpness}
Given $p>n$, $m \in \{1,\ldots,n-1\}$, $V \in G(n,m)$, and a set $A \subset V^\perp$ such that $\dim_H A \ge \beta(\alpha)$, there exists a continuous map $f \in W^{1,p}(\R^n:\R^n)$ so that $\dim_H f(V_\ba \cap \Omega) \ge \alpha$ for every $\ba \in A$.
\end{theorem}

The proof of Theorem \ref{th:bmt-foliations} in \cite{bmt:foliations} couples Kaufman's argument with the volume growth version of Frostman's lemma to obtain almost sure conclusions. The proof is streamlined by observing that the usual dyadic decomposition of $\R^n$ is well-adapted to the orthogonal splitting of $\R^n=V \oplus V^\perp$. Arguing by contradiction, let $\Exc(\alpha) := \{ \ba \in V^\perp : \cH^\alpha(f(V_\ba \cap \Omega)) > 0 \}$, assume that $\cH^\beta(\Exc(\alpha)) > 0$, and let $\nu$ be a Radon measure supported on $\Exc(\alpha)$ such that $\nu(B_{V^\perp}(\ba,r)) \le r^\beta$ for any $\ba \in V^\perp$ and $r>0$. Implementing the previous argument of Kaufman along parallel fibers of the orthogonal projection $\pi:\R^n \to V^\perp$, one deduces that
$$
\int_{V^\perp} \cH^\alpha_\eps(f(V_\ba \cap \Omega)) \, d\nu(\ba) \to 0 \qquad \mbox{as $\eps \to 0$.}
$$
It follows that $\cH^\alpha(f(V_\ba \cap \Omega)) = 0$ for $\nu$-a.e.\ $\ba$, but this contradicts the choice of the exceptional set $\Exc(\alpha)$ and the fact that $\nu$ is a positive Borel measure on $\Exc(\alpha)$.

Theorem \ref{th:bmt-foliations} applies in particular to quasiconformal homeomorphisms of domains in $\R^n$, where it yields new information about the dimension distortion for generic parallel subspaces. For simplicity, and because results in this case are stated the most cleanly due to Astala's theorem, we only state the relevant theorem in the plane.

\begin{corollary}\label{cor:bmt-qc-plane}
Let $f:\Omega \to \Omega'$ be a $k$-quasiconformal homeomorphism between planar domains, let $L \in G(2,1)$, and let $1<\alpha<1+k$. Then $\dim_H f(L_\ba \cap \Omega) \le \alpha$ for $\cH^\beta$-a.e.\ $\ba \in L^\perp$ and $\beta = 1-(1+\tfrac1k)(1-\tfrac1\alpha)$.
\end{corollary}

Since $k < 1$, we have $\beta < \tfrac2\alpha - 1 < 1 = \dim L^\perp$. In \cite{bmt:foliations}, examples of quasiconformal maps were given demonstrating sharpness of this weaker, dilatation-independent consequence of Corollary \ref{cor:bmt-qc-plane}, but for the box-counting dimension of $f(L_\ba \cap \Omega)$ rather than the Hausdorff dimension. Such examples fail to be sharp for the dilatation-dependent statement in Corollary \ref{cor:bmt-qc-plane}. Indeed, since the stronger bounds for Hausdorff dimensions of planar quasilines are given by Smirnov's theorem \ref{th:haus-dim-quasilines} relative to the bounds in \eqref{eq:qc-bound-quasilines-1}, one in fact anticipates that Corollary \ref{cor:bmt-qc-plane} can be improved.

\begin{question}\label{q:k-dependent}
What are optimal $k$-dependent estimates for the generic distortion of Hausdorff dimension of parameterized families of parallel lines under planar $k$-quasiconformal maps?
\end{question}

Rather than focusing on dilatation, one may instead seek quasiconformal maps which are sharp from the perspective of Sobolev membership, consistent with the bounds in Proposition \ref{prop:kaufman}. (Note that this question is independent of the determination of the sharp Sobolev higher integrability exponent for quasiconformal mappings in higher dimensions.) In \cite{btw:increase}, Balogh, Wildrick and the author construct quasiconformal mappings, in arbitrary dimensions, illustrating such sharpness.

\begin{theorem}[Balogh--Tyson--Wildrick]\label{th:btw-increase}
Fix $p>n\ge 2$, $1 < \alpha < \tfrac{p}{p-n+1}$, and $L \in G(n,1)$. Then for any $\eps>0$ there exists $A \subset L^\perp$ with $\dim_H A > (n-1) - p(1-\tfrac1\alpha) - \eps$ and a QC map $f$ in $W^{1,p}(\R^n:\R^n)$ so that $\dim_H f(L_\ba) \ge \alpha$ for each $\ba \in A$.
\end{theorem}

The construction draws inspiration from Example \ref{ex:gv}, but requires a small adjustment for the sharpness. To simplify the exposition we set $n=2$. One again works with families of disjoint closed squares located along the fibers of $P_{L^{\perp}}$, but in order to increase dimension in the target by close to the optimal amount one must pack the image squares as closely as possible. Such packing cannot always be done by squares for simple number-theoretic considerations. However, we may replace cubes in the target with rectangles oriented along the $x$-axis, and switch to rectangles oriented along the $y$-axis at the second stage of the iteration. After two stages we again have a square in the target and the self-similar iteration proceeds as before. With this simple idea, it remains to carry out the technical estimates to show that the optimal choices for $\dim_H A$ vis-\`a-vis $\dim_H f(L_\ba)$, $\ba \in A$, and the Sobolev exponent $p$ are attained.

Bishop, Hakobyan and Williams \cite{bhw:ahlfors} gave an earlier example in the case $n=2$. The example in \cite{bhw:ahlfors} exhibits more striking distortion along the parallel lines $L_\ba$, namely, $\dim_H f(F) = \alpha \dim_H F$ for every $\ba \in A$ and every Borel $F \subset L_\ba$. Thus $f|_{L_\ba}$ behaves similar to a snowflake mapping for every choice of $\ba \in A$. The construction of such a map uses heavily the planar setting, relying on conformal mapping tools. The authors of \cite{bhw:ahlfors} also prove a number of other intriguing results about quasiconformal dimension distortion. For instance, here is a special case of \cite[Theorem 1.1]{bhw:ahlfors}.

\begin{theorem}[Bishop--Hakobyan--Williams]
Let $f:\R^n \to \R^n$ be quasiconformal, $n \ge 2$, and let $E \subset \R^n$ be the invariant set for a self-similar itferated function system satisfying the open set condition. Then $\dim_H f(E+\by) = \dim_H E$ for $\cL^n$-a.e.\ $\by \in \R^n$.
\end{theorem}


\section{Conformal dimension}\label{sec:conformal-dimension}

In this section, we provide a short introduction to the theory of conformal dimension, emphasizing its relevance for Euclidean quasiconformal dimension distortion. The following definition was introduced by Pansu in 1989 \cite{pan:conformal-dimension}. We first note that the notion of Hausdorff dimension $\dim_H X$ for a metric space $(X,d)$ is defined exactly as in section \ref{subsec:basic-definitions-and-notation}, using coverings of $X$ by subsets $A_i$. Similarly, the box-counting dimension $\dim_B X$ is well-defined for totally bounded metric spaces $(X,d)$ via the covering number $N(X,r)$, $r>0$.

\begin{definition}[Pansu]\label{def:conf-H-dim}
The {\it conformal (Hausdorff) dimension} of $(X,d)$, $C\dim_H X$, is the infimum of the values $\dim_H Y$ over all spaces $(Y,d')$ QS equivalent to $(X,d)$.
\end{definition}

In view of the role played by quasisymmetric maps in applications of analysis in metric spaces (e.g., to geometric group theory, dynamics on nonsmooth spaces, or analysis on fractals), {\it quasisymmetric invariants} are of great interest. Conformal dimension is among the most widely studied of these invariants. For many applications, a variant notion has received greater attention. We begin with a definition.

\begin{definition}
The {\it Assouad dimension} of $(X,d)$, $\dim_A X$, is the infimum of $s > 0$ so that there exists $C>0$ with $N(B(x,R),r) \le C(R/r)^s$ for all $x \in X$, $0<r<R$.
\end{definition}

A metric space $(X,d)$ is {\it doubling} if there exists $C>0$ so that $N(B(x,2r),r) \le C$ for all $x \in X$ and $r>0$. It is easy to see that $X$ is doubling if and only if $\dim_A X < \infty$. Assouad studied this notion of dimension (but under a different name) in \cite{ass:plongements} in connection with metric embedding problems. A key result from \cite{ass:plongements} is that $(X,d)$ QS embeds\footnote{In fact, Assouad proves that the snowflaked metric space $(X,d^\eps)$ bi-Lipschitz embeds into some finite-dimensional Euclidean space for any $0<\eps<1$. Compare the discussion in Example \ref{ex:snowflake}.} into some finite dimensional Euclidean space if and only if $\dim_A X < \infty$. 

It is instructive to compare Assouad dimension to box-counting dimension; both notions ask for uniform control on covering numbers $N(E,r)$, however, box-counting dimension only requires control on a global scale ($E=X$), while Assouad dimension requires uniform control across all locations and scales. The inequalities $\dim_H X \le \dim_B X \le \dim_A X$ for bounded spaces are clear, moreover, $\dim_H X \le \dim_A X$ for all $X$. All dimensions coincide and equal $n$ for any nonempty open subset of $\R^n$.

\begin{example}\label{ex:f-p}
For $p>0$, let $F_p = \{m^{-p}:m \in \N\} \cup \{0\}$. Then $\dim_H F_p = 0$. By comparing the gap size $\tfrac1{m^p} - \tfrac1{(m+1)^p} \approx \tfrac1{m^{p+1}}$ with the diameters of tails of the sequence, $\diam \{ j^{-p}:j \ge m\} \cup \{0\} = m^{-p}$, one concludes that $\dim_B F_p = \tfrac{1}{1+p}$. To compute the Assouad dimension, we study weak tangent spaces. Using Proposition \ref{prop:conformal-dimension-and-weak-tangents} below and the fact that the interval $[0,\infty)$ occurs as a weak tangent space for $F_p$, we find that $1 = \dim_A [0,\infty) \le \dim_A F_p \le 1$. Hence $\dim_A F_p = 1$.
\end{example}


In Example \ref{ex:f-p} we used a basic fact about Assouad dimension: it contracts under Gromov--Hausdorff convergence. Similar conclusions do not hold for the other notions of dimension considered in this survey. A pointed metric space $(Z,d_Z,z_0)$ is said to be a {\it weak tangent} of $(X,d_X)$ if there exist scales $r_m \to 0$ and points $x_m \in X$ so that $(X,r_m^{-1} d_X,x_m)$ converges (in the pointed Gromov--Hausdorff sense) to $(Z,d_Z,z_0)$. We refer to \cite[Chapter 6]{mt:conformal-dimension} for the definition of Gromov--Hausdorff convergence and its pointed variant. The following result is \cite[Proposition 6.1.5]{mt:conformal-dimension}.

\begin{proposition}\label{prop:conformal-dimension-and-weak-tangents}
Let $(Z,d_Z)$ be a weak tangent of $(X,d_X)$. Then $\dim_A Z \le \dim_A X$.
\end{proposition}

By analogy with Definition \ref{def:conf-H-dim}, the {\it conformal Assouad dimension} of $(X,d)$, $C\dim_A X$, is the infimum of $\dim_A Y$ over all metric spaces $(Y,d')$ QS equivalent to $(X,d)$. By an observation of Heinonen \cite[Theorem 14.16]{hei:lams}, for metric spaces $(X,d)$ which are both doubling and uniformly perfect,\footnote{$(X,d)$ is {\it uniformly perfect} if there exists $c>0$ so that for all $x \in X$ and all $0<r<\diam X$, $B(x,r) \setminus B(x,cr) \ne \emptyset$.} $C\dim_A X$ coincides with the {\it Ahlfors regular conformal dimension} $ARC\dim X$ equal to the infimum of all $Q$ for which there exists an Ahlfors $Q$-regular space $(Y,d')$ QS equivalent to $(X,d)$. Recall that a measure $\nu$ on $(X,d)$ is {\it Ahlfors $Q$-regular} if there exist $C_1>0$ and $C_2<\infty$ so that $C_1 r^Q \le \nu(B(x,r)) \le C_2 r^Q$ for all $x \in X$ and $0<r<\diam X$. A metric space $(X,d)$ is Ahlfors $Q$-regular if it supports an Ahlfors $Q$-regular measure, and every Ahlfors regular space is both doubling and uniformly perfect.

For sets $E \subset \R^n$, we may consider instead the {\it global quasiconformal dimension} $\GQC_{\R^n}\dim E := \inf \{ \dim f(E) : \mbox{$f:\R^n \to \R^n$ QC} \}$. Here $\dim \in \{\dim_H,\dim_B,\dim_A\}$.

\begin{remark}
For any $(X,d)$, the {\it snowflake map} $(X,d) \to (X,d^\alpha)$, $0<\alpha<1$, is QS and $\dim(X,d^\alpha) = \alpha^{-1} \dim(X,d)$. (Here $\dim$ may refer to any of the above notions.) Thus if $\dim (X,d)$ is positive, then the dimension can be increased arbitrarily by a QS map. With regards to global quasiconformal mappings, Bishop \cite{bis:increase} showed that if $E \subset \R^n$ is closed with $\dim_H E > 0$, then $\sup \{\dim_H f(E) : \mbox{$f$ QS map of $\R^n$} \} = n$. To the best of our knowledge the following questions remain open.

\begin{question}
Let $\dim \in \{\dim_B,\dim_A\}$ and let $E \subset \R^n$ be a closed set with $\dim E > 0$. Is it true that $\sup \{ \dim f(E) : \mbox{$f$ QS map of $\R^n$} \} = n$?
\end{question}
\end{remark}

A set $E$ is said to be {\it minimal for conformal dimension} if $C\dim E = \dim E$. The following example comes from Bishop--Tyson \cite[Remark 1]{bt:locally-minimal}.

\begin{example}\label{ex:bar-code}
For any Borel set $Z \subset \R^{n-1}$, the set $E = Z \times [0,1]$ is minimal for conformal Hausdorff dimension.
\end{example}

By way of contrast, note that $E = Z \times F_p$ is minimal for conformal Assouad dimension, for any $p>0$ and any Ahlfors regular $Z \subset \R^{n-1}$. In fact, every QS map between doubling metric spaces descends to a QS map between weak tangents. Since $E$ has weak tangents of the form $Z_\infty \times [0,1]$ and $Z_\infty$ is again Ahlfors regular with $\dim Z_\infty = \dim Z$, we conclude that $C\dim_A E \ge C\dim_A (Z_\infty \times [0,1]) \ge C\dim_H (Z_\infty \times [0,1]) = \dim Z_\infty + 1 = \dim Z + 1 = \dim_A E$.

In \cite{tys:assouad} the author proved that no spaces with Assouad dimension in $(0,1)$ are minimal for conformal Assouad dimension. More precisely, if $(X,d)$ has $\dim_A X < 1$, then $C\dim_A X = 0$. The corresponding question for Hausdorff or box-counting dimension was resolved by Kovalev \cite{kov:duke}.

\begin{theorem}[Kovalev]\label{th:kovalev-theorem}
Let $\dim \in \{ \dim_H, \dim_B \}$ and let $(X,d)$ be a metric space with $\dim X < 1$. Then $C\dim X = 0$.
\end{theorem}

In fact, Kovalev proves a stronger conclusion: whenever $X$ is isometrically embedded into some Banach space $V$, then one can choose QS self-maps of $V$ which reduce the dimension of $X$ arbitrarily. In particular, if $E \subset \R^n$ has $\dim E < 1$, then $\GQC\dim_{\R^n} E = 0$.

We highlight some known results and open questions on conformal dimensions of self-similar Euclidean sets. The {\it Sierpi\'nski gasket} $\SG$ (Figure \ref{fig:gaskets}(a)) is the invariant set for a planar IFS consisting of three contractions each with scale factor $\tfrac12$.   The dimension of $\SG$ is $\tfrac{\log 3}{\log 2}$. In \cite{tw:gasket}, the author and J.-M. Wu showed that $\GQC\dim_{\R^2} \SG = 1$. Moreover, the infimum is not achieved\footnote{Any QC map of $\R^2$, or more generally any QS map from $\SG$ to any other metric space, has image set with dimension strictly greater than one.} and can be reached via invariant sets of self-similar IFS. See Figure \ref{fig:gaskets}(b) for an example of one such deformed IFS.\footnote{For an animation of a one-parameter deformation of $\SG$ through the moduli space of invariant sets of self-similar IFS, see {\tt https://publish.illinois.edu/jeremy-tyson/animations}.} The construction generalizes to a more extensive class of post-critically finite iterated function systems satisfying suitable geometric constraints on the location of, and relationship between, the self-similar pieces. We refer to \cite{tw:gasket} for more details.

\begin{figure}
\includegraphics[scale=.36]{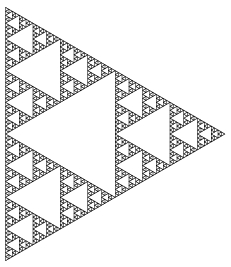} \hspace{20pt} \includegraphics[scale=.45]{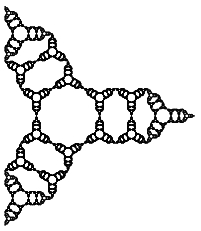}
%
%
\caption{(a) The Sierpi\'nski gasket $\SG$; (b) a QS deformation of $\SG$, given as the invariant set for a new self-similar IFS.}
\label{fig:gaskets}
\end{figure}

By way of contrast, determining the conformal dimension of the {\it Sierpi\'nski carpet} $\SC$ is a longstanding open problem. The carpet $\SC$ is the invariant set for a planar IFS consisting of eight contractions each with scale factor $\tfrac13$, and hence $\dim \SC = \tfrac{\log 8}{\log 3} \approx 1.89279\ldots$. Unlike $\SG$, the defining IFS is not post-critically finite. Known facts about $C\dim \SC$ include:
\begin{itemize}
\item $C\dim_A \SC < \dim \SC$. This fact follows from a general result of Keith and Laakso \cite{kl:conf-a-dim} characterizing which Ahlfors regular spaces are minimal for conformal Assouad dimension.
\item $C\dim_H \SC = C\dim_A \SC$. This is a consequence of Theorem \ref{th:eb} below. We denote this value by $C\dim \SC$.
\item $C\dim \SC \ge 1 + \tfrac{\log 2}{\log 3} \approx 1.63093\ldots > 1$. This follows from the inclusion $\SC \supset C \times [0,1]$, where $C$ denotes the standard Cantor set, together with Example \ref{ex:bar-code}.
\item $1.7652\ldots \le C\dim \SC \le 1.8067\ldots$. These numeric bounds were obtained in a computer-assisted calculation by Kwapisz \cite{kw:carpet}, using estimates for $p$-resistance of approximating networks. Kwapisz's upper bound improves an earlier upper estimate by Kigami. The paper \cite{kw:carpet} also suggests a conjectural construction of the optimal (uniformizing) metric on $\SC$.
\end{itemize}
To the best of the author's knowlege, it is not known whether or not $\GQC\dim_{\R^2} \SC$ is equal to $\dim \SC$. In other words, does there exist any QC map $f:\R^2 \to \R^2$ so that $\dim f(\SC) < \dim \SC$? 

\begin{example}\label{ex:bedford-mcmullen}
Bedford--McMullen carpets are a canonical class of self-affine but non-self-similar fractal sets which serve as a model for general self-affine geometry. The Hausdorff and box-counting dimensions of such carpets were computed independently by Bedford and McMullen in the 1980s.  Here we only want to mention an interesting dichotomy proved by Mackay \cite{mac:carpets}: every such carpet $E \subset \R^2$ is either minimal for conformal Assouad dimension or has conformal Assouad dimension zero. Mackay also gives an explicit formula for $\dim_A E$ for these carpets. The conformal Hausdorff dimensions of Bedford--McMullen carpets are still not known.

We refer the interested reader to Fraser's survey article \cite{fr:bedford-mcmullen-survey} for a broad overview of Bedford-McMullen carpets and tools used to compute or estimate dimensions.
\end{example}

We conclude this section with a recent result of Eriksson-Bique \cite{seb:conformal-dimension}. We refer to \cite{seb:conformal-dimension} for the definition of {\it quasi-self-similar metric space}, but note that it includes all invariant sets for self-similar IFS satisfying the open set condition.

\begin{theorem}[Eriksson-Bique]\label{th:eb}
$C\!\dim_A X = C\!\dim_H X$ for quasi-self-similar $(X,d)$.
\end{theorem}

Theorem \ref{th:eb} indicates that results about other variants of conformal dimension (e.g., conformal box-counting dimension or the conformal versions of the dimension interpolants studied in the following section) only provide nontrivial information in case the underlying space is not self-similar.

\subsection{References}

The literature on conformal dimension is vast, and in the interest of brevity we omit discussion of many other interesting topics and applications. The monograph \cite{mt:conformal-dimension} contains an extensive discussion of conformal dimension in metric spaces, along with applications to geometric group theory. For more information about the role of conformal dimension in connection with Cannon's conjecture and dynamics on fractal spheres, we recommend the ICM lecture by Bonk in \cite{bonk:icm}.
Kwapisz's estimates for the conformal dimension of $SC$ in \cite{kw:carpet} rely on characterizations of conformal dimension as a critical exponent for certain discrete moduli on graph approximations. This idea has been explored by many authors in recent years; for a partial list of references we mention \cite{bp:combinatorial}, \cite{hp:cxc}, \cite{cp:gauge}, and \cite{mur:combinatorial}.

\section{Sobolev and QC distortion of dimension interpolants}\label{sec:interpolation}

In this final section, we describe recent work of the author, in collaboration with Chrontsios Garitsis \cite{ct:qc-assouad-spectrum} and separately with Fraser \cite{ft:Sobolev-intermediate-dimension}, on the quasiconformal and Sobolev distortion properties of dimension interpolants.

{\it Dimension interpolation} refers to a research agenda which identifies geometrically natural one-parameter families of dimensions interpolating between existing notions of metric dimension. Two dimension interpolants feature prominently in this section. The {\it Assouad spectrum}, introduced by Fraser and Yu in \cite{fy:assouad-spectrum}, interpolates between box-counting and Assouad dimension, while the {\it intermediate dimensions}, introduced by Falconer, Fraser, and Kempton in \cite{ffk:intermediate-dimensions}, interpolate between Hausdorff and box-counting dimensions. In Definitions \ref{def:assouad-spectrum} and \ref{def:intermediate-dimensions}, we assume that $E \subset \R^n$ is a bounded set and we fix $0<\theta<1$. Recall that $N(F,r)$, for a bounded set $F$, denotes the minimal number of sets of diameter at most $r$ needed to cover $F$.

\begin{definition}[Assouad spectrum]\label{def:assouad-spectrum}
The {\it $\theta$-Assouad spectrum} $\dim_A^\theta E$ is the infimum of those values $s>0$ for which there exists a constant $C>0$ so that $N(B(x,R) \cap E,r) \le C (R/r)^s$ for balls $B(x,R)$ in $\R^n$ with $x \in E$ and $0<r\le R^{1/\theta} < R < 1$.
\end{definition}

\begin{definition}[Intermediate dimension]\label{def:intermediate-dimensions}
The {\it (upper) $\theta$-intermediate dimension} $\dim_\theta E$ is the infimum of $s>0$ such that for some $C>0$ and all $\delta\in(0,\delta_0(\eps))$, $E$ can be covered by sets $\{A_i\}$ with $\delta^{1/\theta} \le \diam A_i \le \delta$ for all $i$ and $\sum_i (\diam A_i)^s  \le C$.
\end{definition}

Definition \ref{def:assouad-spectrum} differs from the concept defined in \cite{fy:assouad-spectrum}, which considers the covering number $N(B(x,R) \cap E,r)$ for pairs $r<R$ satisfying $r = R^{1/\theta}$ rather than $r \le R^{1/\theta}$. The two notions are in general distinct, but are related: see \cite[Theorem 3.3.6]{fr:book} for a formula relating the resulting two versions of Assouad spectrum. The quantity in Definition \ref{def:assouad-spectrum} was termed {\it upper $\theta$-Assouad spectrum} in \cite[Section 3.3.2]{fr:book} and {\it regularized $\theta$-Assouad spectrum} in \cite{ct:qc-assouad-spectrum}. Definition \ref{def:assouad-spectrum} is more suited to the study of mapping properties, since in general we only obtain an estimate for the diameter of the image of a set under a given map rather than an exact formula.

\begin{example}\label{ex:f-p-2}
For the sets $F_p$ in Example \ref{ex:f-p}, the formulas $\dim_A^\theta F_p = \frac1{(1+p)(1-\theta)} \wedge 1$ and $\dim_\theta F_p = \frac{\theta}{\theta+p}$ were obtained in \cite[Corollary 6.4]{fy:assouad-spectrum} and \cite[\S 3.1]{ffk:intermediate-dimensions}, respectively.
\end{example}

\begin{example}
Higher dimensional analogs for the sets $F_p$ were considered in \cite{bf:ifs} and \cite{ft:Sobolev-intermediate-dimension}. Given $p>0$ and $n \in \N$, set $(\N^p)^n := \{ (m_1^p,\ldots,m_n^p):m_j \in \N \}$. For $\alpha \ne 0$ denote by $f_\alpha:\R^n \setminus\{0\} \to \R^n \setminus \{0\}$ the radial stretch map $f_\alpha(x) = |x|^{\alpha-1} x$. 
The map $f_{-1}(x) = \tfrac{x}{|x|^2}$ is the conformal inversion in $\Sph^{n-1} = \{ x : |x|=1\}$, while in general $f_\alpha$ is $K$-quasiconformal with $K = \max\{|\alpha|,|\alpha|^{-1}\}$. Set $E_{p,n}^\alpha := \overline{f_\alpha((\N^p)^n)}$ for $-1\le\alpha<0$; see Figure \ref{fig:e-p-n-alpha-s} for illustrations. 

\begin{figure}[h]
\includegraphics[scale=.35]{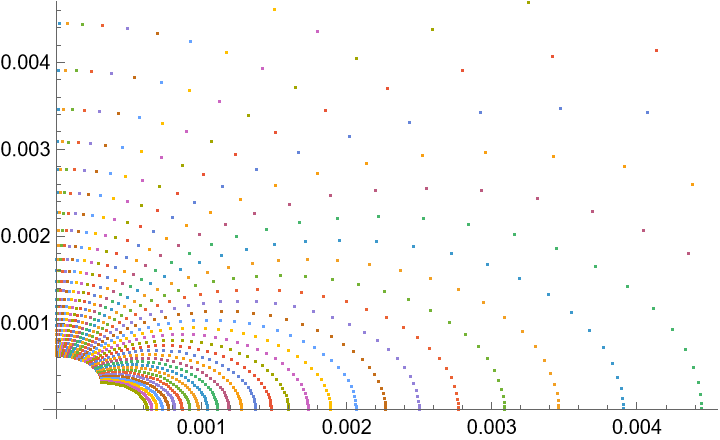} \hspace{40pt} \includegraphics[scale=.35]{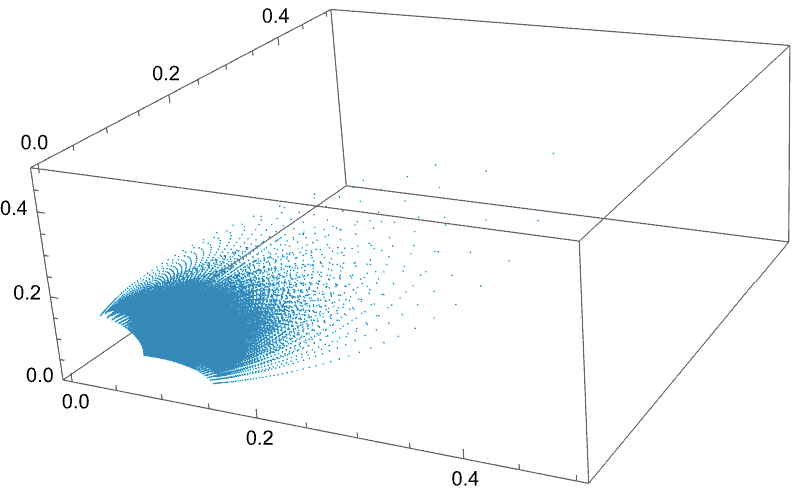}
\caption{Examples of multi-dimensional polynomial sequence sets $E_{p,2}^\alpha$ (left) and $E_{p,3}^\alpha$ (right)}
\label{fig:e-p-n-alpha-s}
\end{figure}

The following expression for the intermediate dimensions was obtained in \cite{bf:ifs} ($\alpha=-1$) and in \cite{ft:Sobolev-intermediate-dimension} (general $\alpha$):
\begin{equation}\label{eq:e-p-n-alpha-intermediate-dimensions}
\dim_\theta E_{p,n}^\alpha = \frac{n \theta}{\theta + p |\alpha|}, \qquad 0<\theta<1, -1\le \alpha<0.
\end{equation}
Hence $\dim_B E_{p,n}^\alpha = \tfrac{n}{1+p|\alpha|}$. Moreover, $\dim_A E_{p,n}^\alpha = n$ for all $p$ and $\alpha$. 
\end{example}

\begin{example}\label{ex:spirals}
Polynomial spirals were studied by Fraser \cite{fr:spirals} in connection with the bi-H\"older winding problem. Given $p>0$, let $S_p := \{ x^{-p} e^{\bi x} : 1<x<\infty \} \cup \{0\} \subset \C$. Theorem 4.4 of \cite{fr:spirals} provides the following formula for the Assouad spectrum of $S_p$: $\dim_A^\theta S_p = \tfrac2{(1+p)(1-\theta)} \wedge 2$ if $0<p<1$, and $\dim_A^\theta S_p = \left( 1+ \tfrac{\theta}{p(1-\theta)} \right) \wedge 2$ if $p \ge 1$. In particular, $\dim_B S_p = \tfrac2{1+p}$ if $0<p<1$ and $\dim_B S_p = 1$ if $p\ge 1$. Clearly $\dim_H S_p = 1$, and another analysis using weak tangents reveals that $\dim_A S_p = 2$. 
\end{example}

\begin{example}\label{ex:bedford-mcmullen-dimension-interpolants}
The exact values of dimension interpolants for Bedford--McMullen carpets have been computed by Fraser and Yu \cite[Theorem 3.3]{fy:fractal-families} (Assouad spectrum) and Banaji and Kolossv\'ary \cite{bk:carpets} (intermediate dimensions).
\end{example}

Propositions \ref{prop:assouad-spectrum-properties} and \ref{prop:intermediate-dimension-properties} collect known facts about these dimension interpolant functions. Of particular importance are inequalities \eqref{eq:Assouad-spectrum-upper-bound} and \eqref{eq:intermediate-dimension-lower-bound} which relate different notions of dimension. For proofs, we refer to \cite{fy:assouad-spectrum}, \cite{ffk:intermediate-dimensions}, and \cite{br:intermediate-dimensions}, see also \cite[Section 3.3]{fr:book} for these and other facts about Assouad spectrum.

\begin{proposition}\label{prop:assouad-spectrum-properties}
Let $E \subset \R^n$ be a bounded set. Then
\begin{enumerate}
\item the function $\theta \mapsto \dim_A^\theta E$ is non-decreasing and continuous in $\theta$ for $0<\theta<1$,
\item $\dim_A^\theta E \searrow \dim_B E$ as $\theta \searrow 0$,
\item if $f:E \to \R^N$ is injective and $\alpha$-H\"older with $\beta^{-1}$-H\"older inverse for some $0<\alpha \le 1 \le \beta < \infty$, then\footnote{provided all of the relevant Assouad spectrum parameters lie in $(0,1)$} $\tfrac{1-\beta\theta/\alpha}{\beta(1-\theta)} \dim_A^{\beta\theta/\alpha} E \le \dim_A^\theta f(E) \le \frac{1-\alpha\theta/\beta}{\alpha(1-\theta)} \dim_A^{\alpha\theta/\beta} E$,
\item for each $0<\theta<1$,
\begin{equation}\label{eq:Assouad-spectrum-upper-bound}
\dim_A^\theta E \le \frac{\dim_B X}{1-\theta} \wedge \dim_A E.
\end{equation}
\end{enumerate}
\end{proposition}

\begin{proposition}\label{prop:intermediate-dimension-properties}
Let $E \subset \R^n$ be a bounded set. Then
\begin{enumerate}
\item the function $\theta \mapsto \dim_\theta E$ is non-decreasing and continuous in $\theta$ for $0<\theta<1$,
\item $\dim_A^\theta E \nearrow \dim_B E$ as $\theta \nearrow 1$,
\item $\dim_\theta f(E) \le \dim_\theta E$ for $f:E \to \R^N$ Lipschitz and $0<\theta<1$, ,
\item $\dim_\theta f(E) \le \tfrac1\alpha \dim_\theta E$ for $f:E \to \R^N$ $\alpha$-H\"older, $0<\alpha<1$, and $0<\theta<1$,
\item for each $0<\theta<1$,
\begin{equation}\label{eq:intermediate-dimension-lower-bound}
\dim_\theta E \ge \frac{\theta(\dim_A E)(\dim_B E)}{\dim_A E - (1-\theta) \dim_B E}.
\end{equation}
\end{enumerate}
\end{proposition}

In particular, it follows from Proposition \ref{prop:intermediate-dimension-properties}(5) that if $\dim_B E = \dim_A E$ then $\dim_\theta E = \dim_B E$ for all $\theta>0$.

We record the following chain of inequalities valid for any bounded set $E \subset \R^n$:
$$
\dim_H E \le \dim_{\theta'} E \le \dim_B E \le \dim_A^\theta E \le \dim_A E, \qquad \forall \, 0<\theta,\theta'<1.
$$
The limit of $\dim_A^\theta E$ as $\theta \nearrow 1$ exists, but may not equal $\dim_A E$. In fact, $\dim_{qA} E = \lim_{\theta \nearrow 1} \dim_A^\theta E$ is the so-called {\it quasi-Assouad dimension} of $E$ \cite[Section 3.2]{fr:book}. Similarly, $\dim_{qH} E = \lim_{\theta \searrow 0} \dim_\theta E$ is the {\it quasi-Hausdorff dimension} of $E$. 

Parts (3) and (4) of Proposition \ref{prop:intermediate-dimension-properties} assert that intermediate dimension behaves nicely with respect to both Lipschitz and H\"older maps. However, the Assouad spectrum need not satisfy such estimates, cf.\ \cite[Theorem 3.4.12]{fr:book}. Moreover, no general bounds for Sobolev distortion of Assouad spectrum are known; note that if an estimate similar to \eqref{eq:kaufman} held for $\dim_A^\theta$, then passing to the limit as $p \to \infty$ would yield Lipschitz contractivity of $\dim_A^\theta$. However, analogs of the QC distortion estimates \eqref{eq:qc-haus-dim-distortion-b} do hold true. The following theorem comes from \cite{ct:qc-assouad-spectrum}; see Definition \ref{def:rh-exponent} for the definition of $p^\RH(n,K)$.

\begin{theorem}[Chrontsios Garitsis--Tyson]\label{th:ct-qc-assouad-spectrum}
For $p>n \ge 2$ and $t>0$, set $\alpha_n(p) = 1 - \tfrac{n}{p}$, $\Phi_n(s) = s^{-1} - n^{-1}$, and $\theta(t) = \tfrac1{1+t}$. If $f:\Omega \to \Omega'$ is $K$-QC in $\R^n$, then
\begin{equation}\begin{split}\label{eq:ct-qc-assouad-spectrum}
&{\alpha_n(p^\RH(n,K))} \Phi_n(\dim_A^{\theta(t/K)}(E)) \le \Phi_n(\dim_A^{\theta(t)} f(E)) \\
&\qquad \le {\alpha_n(p^\RH(n,K^{n-1}))^{-1}} \Phi_n(\dim_A^{\theta(Kt)}(E)).
\end{split}\end{equation}
for any compact $E\subset \Omega$. In particular, if $n=2$ then
\begin{equation}\label{eq:ct-qc-assouad-spectrum-planar}
K^{-1} \Phi_2(\dim_A^{\theta(t/K)}(E)) \le \Phi_2(\dim_A^{\theta(t)} f(E)) \le K \Phi_2(\dim_A^{\theta(Kt)}(E)).
\end{equation}
\end{theorem}

Several features of Theorem \ref{th:ct-qc-assouad-spectrum} require further discussion. First, the quality of the distortion estimates is governed by a new higher integrability exponent. In Gehring's proof in \cite{geh:higher} of his higher integrability theorem, a key step was to prove that the Jacobian of a $K$-QC map in $\R^n$ satisfies a {\it reverse H\"older inequality} of the following type: there exists $C>0$ so that for every ball $B$ such that $2B \Subset \Omega$,
\begin{equation}\label{eq:RHI}
\left( \frac1{|2B|} \int_{2B} |Df|^p \right)^{1/p} \le C \left( \frac1{|B|} \int_B |Df|^n \right)^{1/p}.
\end{equation}
Since every QC map already lies in $W^{1,n}_\loc$, it immediately follows that in fact $f \in W^{1,p}_\loc$.

\begin{definition}\label{def:rh-exponent}
For $K \ge 1$ and $n \ge 2$, the {\it reverse H\"older inequality higher integrability exponent} $p^{\RH}(n,K)$ is the supremum of $p>n$ so that \eqref{eq:RHI} holds true for each $K$-QC map between domains in $\R^n$.
\end{definition}

It is clear that $p^\RH(n,K) \le p^\Sob(n,K)$, and conjecturally equality holds throughout. When $n=2$ Astala's proof shows that $p^\RH(2,K) = p^\Sob(2,K) = \tfrac{2K}{K-1}$ which implies the estimates in \eqref{eq:ct-qc-assouad-spectrum-planar}.

Quasiconformal maps are bi-H\"older continuous. The bounds in \eqref{eq:ct-qc-assouad-spectrum-planar} are stronger than those obtained by using Proposition \ref{prop:assouad-spectrum-properties}(4) and known H\"older regularity properties of quasiconformal maps.

Passing to the limit as $t \to 0$ in \eqref{eq:ct-qc-assouad-spectrum} yields similar conclusions for quasi-Assouad dimension, and an easy variation yields similar conclusions for Assouad dimension. In fact,\footnote{This observation is due to Ntalampekos.} using the fact that $\dim_A E = \dim_H Z$ for some weak tangent $Z$ of $E$ \cite[Proposition 5.8]{kor:carpets} and Theorem \ref{th:qc-haus-dim-distortion}, together with the observation that QC maps pass to weak tangents without increase of the dilatation $K$, one may derive the following:

\begin{theorem}\label{th:ct-qc-assouad}
If $f:\Omega \to \Omega'$ is $K$-quasiconformal map in $\R^n$ and $E \subset \Omega$, then
${\alpha_n(p^\Sob(n,K))} ( \frac1{\dim_A E} - \frac1n ) \le \frac1{\dim_A f(E)} - \frac1n \le {\alpha_n(p^\Sob(n,K^{n-1}))^{-1}} ( \frac1{\dim_A E} - \frac1n )$.
\end{theorem}

We do not know if such improvement holds for Assouad spectrum, as there is no known analog for the representation via Hausdorff dimensions of weak tangents.

As an application of \eqref{eq:ct-qc-assouad-spectrum-planar} and Example \ref{ex:spirals}, we obtain a sharp QC classification theorem for polynomial spirals at the level of dilatation. Observe that the radial stretch map $f(z) = |z|^{1/K-1} z$ is $K$-QC and satisfies $f(S_p) = S_q$ where $q = \tfrac{p}{K}$.

\begin{theorem}[Chrontsios Garitsis--Tyson]\label{th:x-t}
Fix $p>q>0$. There exists a quasiconformal map $f:\C\to\C$ with $f(S_p) = S_q$ if and only if $H_f \ge \tfrac{p}{q}$.
\end{theorem}

For the proof, we observe that if $H_f < \tfrac{p}{q}$ then, by making a suitable choice of $t>0$ and using Example \ref{ex:spirals} we obtain that $f^{-1}$ maps a spiral with $\theta(t)$-Assouad spectrum equal to $2$ onto a spiral with $\theta(t/K)$ spectrum strictly less than $2$. This contradicts \eqref{eq:ct-qc-assouad-spectrum-planar}. For a higher-dimensional analog of Theorem \ref{th:x-t}, involving the quasiconformal classification of polynomially decreasing sequences of spiral shells via the $n$-dimensional assertions in Theorem \ref{th:ct-qc-assouad-spectrum}, see \cite{x:shells}.

Since both Hausdorff and box-counting dimension satisfy the standard upper bounds for the increase under continuous, supercritical Sobolev maps, the following result from \cite{ft:Sobolev-intermediate-dimension} is not surprising.

\begin{theorem}[Fraser--Tyson]\label{th:sobolev-distortion-intermediate-dimension}
If $E \Subset \Omega \subset \R^n$, $n \ge 2$, and $f \in W^{1,p}(\Omega:\R^N)$ continuous and supercritical, then $\dim_\theta f(E) \le \tfrac{p \, \dim_\theta E}{p-n+\dim_\theta E}$ for all $0<\theta<1$.

In particular, two-sided bounds of the usual type hold true for quasiconformal mappings, with the sharp constants $1/K$ and $K$ in the planar case.
\end{theorem}

The proof of Theorem \ref{th:sobolev-distortion-intermediate-dimension} synthesizes elements of the Hausdorff and box-counting proofs. The major/minor decomposition of dyadic cubes is further refined by considering cubes whose images do or do not satisfy the appropriate lower bound on diameter needed to ensure that the target collection is admissible for $\theta$-intermediate dimension. For cubes whose image has too small diameter, one replaces the image with a larger set of an appropriate diameter. The desired estimate is obtained by carrying out the Hausdorff dimension-type proof for minor cubes whose image is $\theta$-admissible, and using cardinality estimates \`a la the box-counting-type proof for the remaining collections. For details, see \cite[Section 4]{ft:Sobolev-intermediate-dimension}.

We next state some applications of both the H\"older and Sobolev distortion estimates for intermediate dimension to quasiconformal classification.
\begin{itemize}
\item If $E$ and $F$ are compact sets with $\dim_H E = \dim_{qH} E = 0$ and $0 < \dim_B F = \dim_A F$, then no QC map of $\R^n$ sends $E$ to $F$. To see this, note that Proposition \ref{prop:intermediate-dimension-properties}(5) implies that $\dim_{qH} F > 0 = \dim_{qH} E$, but this conclusion yields a contradiction after letting $\theta \searrow 0$ in Proposition \ref{prop:intermediate-dimension-properties}(5). For example, $F_p$, $p>0$ (considered as a subset of $\R^n$), is QC inequivalent to any $F$ with $0 < \dim_B F = \dim_A F$.
\item Let $E$ and $F$ be compact with $\dim_{qH} E = \dim_{qH} F = 0$ such that $\theta \mapsto \dim_\theta E$ and $\theta \mapsto \dim_\theta F$ are right differentiable at $\theta = 0$ with 
$$
0< (d/d\theta) \dim_\theta E |_{\theta = 0} < (d/d\theta) \dim_\theta F |_{\theta = 0} = + \infty.
$$
Then no QC map of $\R^n$ sends $E$ to $F$. 
\item Banaji--Kolossv\'ary \cite{bk:carpets} exhibit totally disconnected Bedford--McMullen carpets $E$ and $F$ with $\dim E = \dim F$ for $\dim \in \{\dim_H,\dim_B,\dim_A\}$, but different intermediate dimension profiles. Using the latter, they conclude that $E$ and $F$ are not equivalent by any bi-H\"older homeomorphism with H\"older exponent sufficiently close to one. Using Theorem \ref{th:sobolev-distortion-intermediate-dimension} we improve their conclusion: $E$ and $F$ are not equivalent by any $K$-QC map with $K$ sufficiently close to one, moreover, the bound on $K$ obtained is strictly better than the bound obtained using the Banaji--Kolossv\'ary result and known H\"older regularity properties of QC maps.
\end{itemize}

\begin{remark}
The estimate $K^{-1}((\dim_\theta E)^{-1} - \tfrac12) \le (\dim_\theta f(E))^{-1} - \tfrac12$ for compact sets $E \subset \C$ and $K$-QC maps $f$ of $\C$ is sharp. Surprisingly, the proof of this fact is significantly simpler than Astala's proof of sharpness of the corresponding Hausdorff dimension bounds. In fact, the following fact holds true, highlighting an intriguing connection between dimension interpolation theory and quasiconformal maps.
\end{remark}

\begin{proposition}
Let $E \subset \C$ be non-uniformly perfect with $\dim_A E = 2$, and let $f:\C \to \C$ be $K$-quasiconformal so that $(\dim_B f(E))^{-1} - \tfrac12 = \tfrac1K ( (\dim_B E)^{-1} - \tfrac12 )$ (the pair $(E,f)$ is sharp for the planar QC box-counting dimension estimate) and $\dim_\theta E = \tfrac{2\theta \dim_B E}{2 + (1-\theta) \dim_B E}$ for some $0<\theta<1$ ($E$ is sharp for the Banaji--Rutar lower bound \eqref{eq:intermediate-dimension-lower-bound}). Then $(\dim_\theta f(E))^{-1} - \tfrac12 = \tfrac1K ( (\dim_\theta E)^{-1} - \tfrac12 )$.
\end{proposition}

An example of a set to which the previous proposition applies is the planar polynomial sequence set $E = E_{s,2}^\alpha$ for $-1<\alpha<0$. Note that if $-1<\alpha<\beta<0$ then the sets $E_{s,2}^\alpha$ and $E_{s,2}^\beta$ are related by the radial stretch map $f_\gamma$ with $\gamma = |\beta|/|\alpha|$.

Finally, we return to conformal dimension.

\begin{theorem}[Fraser--Tyson]\label{ft:intermediate-conformal-dimension}
Let $E \subset \R^n$ be bounded with $\dim_{qH} E < 1$. Then for each $\theta<1$, it holds that $\inf \{ \dim_A^\theta f(E) \, : \, \mbox{$f:\R^n \to \R^n$ quasiconformal} \} = 0$. In particular, $C\dim_B E = 0$.
\end{theorem}

Examples of sets to which these conclusions apply include the following:
\begin{itemize}
\item $E = Z \times F_p$ for any $p>0$ and $Z \subset \R^{n-1}$ so that $\dim_{qH} Z < 1$ and $\dim_B Z \ge \tfrac{p}{p+1}$,
\item $E = F_p \times F_p \subset \C$ with $0<p\le 1$,
\item $E = F_p \times F_{1/p} \subset \C$ for any $p>0$,
\item $E = E_{s,2}^{-1} \times F_p \subset \R^3$ for any $s$ and $p$ such that $s \le 1 + \tfrac2p$.
\end{itemize}
Moreover, all of the above sets $E$ have $\dim_B E \ge 1$, whence the stated conclusion regarding $C\dim_B E$ does not follow from Theorem \ref{th:kovalev-theorem}.

We comment on the proof of Theorem \ref{ft:intermediate-conformal-dimension}. First, we prove that each instance $\dim_\theta$ of intermediate dimension satisfies the implication in Kovalev's theorem: if $\dim_\theta E < 1$ then $\inf \{ \dim_\theta f(E):\mbox{$f:\R^n \to \R^n$ QC} \} = 0$. This is done by adapting the proof of Theorem \ref{th:kovalev-theorem}. Next, we use the relationships between dimensions as stated in 
\eqref{eq:Assouad-spectrum-upper-bound} and \eqref{eq:intermediate-dimension-lower-bound}. Assume that $E$ satisfies $\dim_{\theta'} E < 1$ for some $\theta'>0$ and fix $\theta<1$ and $\eps>0$. Pick a QC map $f:\R^n \to \R^n$ so that $\dim_{\theta'} f(E) < \tfrac{\eps (1-\theta)\theta'}{1-(1-\theta')\eps/n}$. Then $\dim_B f(E) < (1-\theta) \eps$ by \eqref{eq:intermediate-dimension-lower-bound} and $\dim_A^\theta f(E) < \eps$ by \eqref{eq:Assouad-spectrum-upper-bound}.

\medskip

To conclude, we summarize in Table \ref{tab:1} the following mapping properties for the various notions of metric dimension introduced in this survey:
\begin{itemize}
\item {\it Lipschitz contractivity:} $\dim f(X) \le \dim X$ for any Lipschitz mapping $f:X \to Y$?
\item {\it H\"older quasicontractivity:} $\dim f(X) \le \alpha^{-1} \dim X$ for any $\alpha$-H\"older mapping $f:X \to Y$?
\item {\it Bi-Lipschitz invariance:} if $X$ is bi-Lipschitz equivalent to $Y$, must $\dim X = \dim Y$?
\item {\it Sobolev distortion bounds:} for $f \in W^{1,p}(\Omega \subset \R^n:\R^N)$ continuous and supercritical, must we have $\dim f(E) \le \tfrac{p \, \dim E}{p-n+\dim E}$ for sets $E \subset \Omega$?
\item {\it Quasiconformal distortion bounds:} for $f:\Omega \to \Omega'$ $K$-QC in $\R^n$, do two-sided bounds of the form 
$$
\alpha_n(p_1) \Phi_n(\dim E) \le \Phi_n(\dim f(E)) \le \alpha_n(p_2)^{-1} \Phi_n(\dim E), \quad \Phi_n(s) = s^{-1}-n^{-1},
$$
hold for $E \subset \Omega$ for suitable $p_1,p_2$ depending on $K$ and $n$?
\end{itemize}

\begin{table}[h]
\caption{Mapping properties of metric dimensions}
\label{tab:1}       
\begin{tabular}{p{4.5cm}p{1cm}p{1cm}p{1cm}p{1cm}p{1cm}}
\hline\noalign{\smallskip}
Dimension & $\dim_H$ & $\dim_\theta$ & $\dim_B$ & $\dim_A^\theta$ & $\dim_A$ \\
\\
Lipschitz contractivity & \checkmark & \checkmark & \checkmark & $\cX$ & $\cX$ \\
H\"older quasi-contractivity & \checkmark & \checkmark & \checkmark & $\cX$ & $\cX$ \\
Bi-Lipschitz invariance & \checkmark & \checkmark & \checkmark & \checkmark & \checkmark \\
Sobolev distortion bounds & \checkmark & \checkmark & \checkmark & $\cX$ & $\cX$ \\
QC distortion bounds & \checkmark & \checkmark & \checkmark & \checkmark & \checkmark \\
Higher integrability exponent & $p^\Sob$ & $p^\Sob$ & $p^\Sob$ & $p^\RH$ & $p^\Sob$ \\
\noalign{\smallskip}\hline\noalign{\smallskip}
\end{tabular}
\end{table}

\subsection{References}

The two types of dimension interpolants discussed in this section were introduced in \cite{fy:assouad-spectrum} and \cite{ffk:intermediate-dimensions}. Several excellent survey articles have also appeared in print. We particularly recommend surveys by Fraser \cite{fr:survey} and by Falconer \cite{fal:survey}. The reader interested in more information about Assouad dimension is encouraged to consult Fraser's book \cite{fr:book} for a nice introduction to the subject from the perspective of fractal geometry.

\bibliographystyle{acm}
\bibliography{biblio}
\end{document}